\documentclass[10pt,journal]{IEEEtran}

\usepackage{latexsym,amssymb,amsmath, amsbsy,amsopn, amstext}
\usepackage{graphicx,epsfig,tikz,pgf,hyperref,cancel,color,float,ifpdf}
\usepackage{extarrows,algorithmic,algorithm}
\usepackage{amsmath, amsthm, amssymb,stfloats,algorithm}
\usepackage{subfigure}

\allowdisplaybreaks[4]

\renewcommand{\epsilon}{\varepsilon}
\graphicspath{{./}}

\newcommand{\ubar}{{\bf u}}

\newcommand{\R}{\mathbb{R}}

\newcommand{\ep}{\epsilon}

\newcommand{\C}{\mathcal{C}}
\newcommand{\D}{\mathcal{D}}

\newcommand{\Kzcbf}{K_{\mathrm{zcbf}}}
\newcommand{\Fr}{F_r}

\newcommand{\x}{{\bf x}}

\newtheorem{definition}{\bfseries Definition}
\newtheorem{proposition}{\bfseries Proposition}
\newtheorem{example}{\bfseries Example}
\newtheorem{theorem}{\bfseries Theorem}

\newtheorem{remark}{\bfseries Remark}

\title{Correctness Guarantees for the Composition of Lane Keeping and Adaptive Cruise Control}
\author{Xiangru Xu,\; Jessy W. Grizzle,\; Paulo Tabuada,\;Aaron D. Ames
	\thanks{This research is supported by NSF CPS Award 1239037.}
	\thanks{X. Xu and J.W. Grizzle are with the Department of Electrical Engineering and Computer Science, University of Michigan, Ann Arbor, MI, USA. email: {\tt\small xuxiangr@umich.edu}, {\tt\small grizzle@umich.edu}.}
	\thanks{P. Tabuada is with the Department of Electrical Engineering, University of California, Los Angles, CA, USA. email: {\tt\small tabuada@ucla.edu.}}
	\thanks{A. D. Ames is with the Department of Mechanical and Civil Engineering, California Institute of Technology, Pasadena, CA, USA. email: {\tt\small ames@caltech.edu}.}}

\begin{document}

\maketitle

\begin{abstract}
This paper develops a control approach with correctness guarantees for the simultaneous operation of lane keeping and adaptive cruise control.  The safety specifications for these driver assistance modules are expressed in terms of set invariance. Control barrier functions are used to design a family of control solutions that guarantee the forward invariance of a set, which implies satisfaction of the safety specifications. The control barrier functions are synthesized through a combination of sum-of-squares program and physics-based modeling and optimization.
{\color{black} A real-time quadratic program is posed to combine the control barrier functions with the  performance-based controllers, which can be either expressed as control Lyapunov function conditions or as black-box legacy controllers.  In both cases,  the resulting feedback control guarantees the safety of the composed driver assistance modules in a formally correct manner.
Importantly, the quadratic program admits a closed-form solution that can be easily implemented. The effectiveness of the control approach is demonstrated by simulations in the industry-standard vehicle simulator Carsim.  }\\

\end{abstract}

%


\section{Introduction}\label{sec:intro}
Recent years have witnessed a growing number of safety or convenience modules for automobiles~\cite{campbell2010autonomous,urmson2008autonomous}. Lane keeping, also called active lane keeping or lane keeping assist, is an evolution of lane departure warning~\cite{GerdLanePotential06,talvala2011pushing,huang2016development}, where instead of simply warning of imminent lane departure through vibration of the steering wheel or an audible alarm, the system corrects the vehicle's direction to keep it within its lane. Early systems corrected vehicle direction by differential braking, but current systems actively control steering to maintain lane centering, which is what we will term Lane Keeping (LK) in this paper. Another such system is Adaptive Cruise Control (ACC), which is a driver assistance system that significantly enhances conventional cruise control \cite{ioannou1993autonomous}. When there is no preceding vehicle, an ACC-equipped vehicle maintains a constant speed set by the driver, just as in conventional cruise control; when a preceding vehicle is detected and is driving at a speed slower than the preset speed, an ACC-equipped vehicle changes its control objective to maintaining a safe following distance. {\color{black} ACC uses deceleration/acceleration bounds that are much less than a vehicle's maximal capabilities to ensure driver comfort; when the ACC specification (such as the minimum time-headway) cannot be maintained with comfort-based deceleration bounds, almost all vehicles equipped with ACC either have a warning system or an emergency braking system. Though ACC is legally considered a convenience feature, its specification is still considered as a ``safety constraint'' in this paper, consistent with how an OEM would treat it. }


Worldwide, most of the major manufacturers are now {\color{black}offering passenger 
	vehicles equipped with ACC and LK.} Moreover, these features can be activated simultaneously at highway speeds and require limited driver supervision. In terms of the levels of automation defined by the National Highway Traffic Safety Administration (NHTSA) \cite{NHSTA2014Web}, such vehicles are already at Level 2 (automation of at least two primary control functions designed to work in unison without driver intervention), and they are approaching Level 3 (limited self-driving). The simultaneous control of the longitudinal and lateral dynamics of a vehicle is an important milestone in the bottom-up approach to full autonomy, and therefore, it is crucial to prove that the controllers associated with ACC and LK behave in a formally correct way when they are both activated.


Applying formal methods to the field of (semi-)autonomous driving has attracted much attention in recent
years \cite{asplund2012formal},\cite{seshia2015formal},\cite{forghani2016design}. Particularly, formal correctness guarantees on individual ACC or LK system have been developed by various means. For example, the verification of cruise control systems has been accomplished using formal methods such as satisfiability modulo theory, theorem proving~\cite{loos2011adaptive} and a counter-example guided approach~\cite{stursberg2004verification}; safety guarantees for LK have been established using Lyapunov stability analysis by assuming the longitudinal speed is constant ~\cite{guldner1996analysis,son2015robust} or varying~\cite{GerdLanePotential06,GerdLaneCompa08}.
Furthermore, the so-called correct-by-construction control design, which aims to synthesize controllers to guarantee the closed-loop system satisfies the specification by construction and hence eliminating the need for verification, has also emerged as a viable means of achieving safety. For example, in \cite{petter2016cst}, two provably correct control design methods were proposed for ACC, which rely on fixed-point computations of certain set-valued mappings on the continuous state space or a finite-state abstraction, respectively.


{\color{black}For the simultaneous operation of two (or even more) safety or convenience modules, which are typically coupled through the vehicle's dynamics, it is more challenging to establish correctness guarantees.}
A contract-based design method employing assume-guarantee reasoning is a potential recipe for the compositional design of complex systems~\cite{sangiovanni2012taming,benvenuti2008contract}. Its main idea is to formally define the assumptions and guarantees of each subsystem, which are called contracts among the subsystems, and based on the contracts, to design (or establish) formal guarantees on the overall system. Different types of assume-guarantee formalisms have been proposed, such as temporal logic formulas~\cite{nuzzo2014contract} and supply/demand rates~\cite{kim2015compositional}. The composition of LK and ACC has been studied on the basis of contracts. A passivity-based approach was proposed in \cite{dai2016hscc}, where the ACC and LK dynamics are described as multi-modal port Hamiltonian systems, based on which some energy functions are constructed to prove trajectories of the composed system do not enter a specified unsafe region.
In \cite{stan2016compo}, the system dynamics are represented as discrete-time linear parameter-varying systems, and contracts are established for the variables that couple the two subsystems; controlled-invariant sets are constructed for the ACC and LK subsystems individually to meet the terms of the contracts using an iterative algorithm, such that the overall controller is guaranteed to ensure the safety of the composed system.
Despite these very interesting initial contributions, many {\color{black}safety guarantee} problems on the composition of LK and ACC are still largely open and deserve further investigation.



Turning now to the more general literature on safety specifications, when safety is expressed as set invariance, controlled invariant sets are used to encode both the correct behavior of the closed-loop system and a set of feedback control laws that will achieve it (see~\cite{mareczek2002invariance,wolff2005invariance,kolmanovsky2014reference,vecchio2016LK} and references therein). Under the name of \textit{invariance control}, \cite{mareczek2002invariance} and \cite{wolff2005invariance} extended Nagumo's Theorem to  allow higher-order derivative conditions on the boundary of the controlled invariant set.  As an add-on control scheme, \textit{reference and command governors} utilize the notion of controlled invariance to enforce constraint satisfaction and ensure that the modified reference command is as close  as possible to the original reference command \cite{kolmanovsky2014reference}. In \cite{vecchio2016LK}, characterization of the controlled invariant set for LK was given, which was then used to derive a feedback control strategy to maintain a vehicle in its lane. {\color{black}The common intent of these methods is to construct a controlled invariant set that encodes the safety specifications, and then construct a feedback law that ensures that trajectories of the controlled systems are confined within the set.}

A barrier function (certificate) is another means to prove a safety property of a system based on set invariance. It seeks a function whose sub-level sets (or super-level sets, depending on the context) are all invariant, without the difficult task of computing the system's reachable set~\cite{prajna2004safety,prajna2007framework}.
In \cite{Aronbarriercdc14}, the barrier condition in~\cite{prajna2004safety} was relaxed by only requiring a single super-level set of the function, which represents the safe region, to be invariant. A control barrier function (CBF) extends barrier functions from dynamical systems to control systems.
{\color{black}When CBFs are unified with a performance controller, which can be expressed as a control Lyapunov function (CLF) condition or a legacy controller, through a quadratic programming (QP) framework, safety can be always guaranteed while the performance objective is overridden when safety and performance are in conflict (see  Figure \ref{fig:flow}).
	The QP-based approach was introduced in \cite{Aronbarriercdc14} where it was applied to ACC safety control design. Further work along this line is available in~\cite{Xu2015ADHS,Aakar2015experiment,borrmann2015control,Hsu15Backstepping,quan2016exponential}.}

\begin{figure}[!htb]
	\centering
	\includegraphics[width=6.5cm]{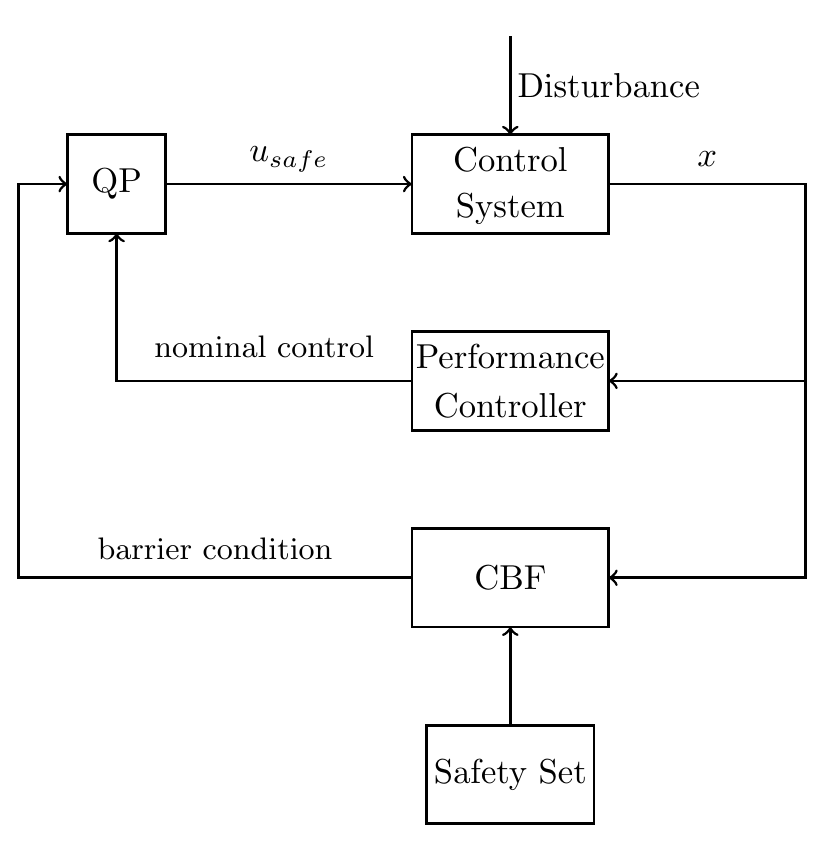}
	\caption{The QP-based framework that unifies the safety constraint and the control performance. The controller $u_{safe}$ is designed to satisfy the barrier condition, which ensures the control system satisfy the safety constraint (i.e., the state $x$ stays within the safety set), and is as close as possible to the given nominal control, which makes the control system achieve the performance objective when it is not conflicting with the safety constraint. }\label{fig:flow}
\end{figure}

This paper develops a modular correct-by-construction control approach that allows for individual and simultaneous activation of LK and ACC, whose dynamics are described by a continuous-time linear system and a linear parameter-varying system, respectively.  The interactions of the two modules through the dynamics of the vehicle and the environment are captured in a ``contract'', which consists of a set of assumptions and guarantees. CBFs are constructed to respect the contract where the CBF for LK is synthesized with the help of  sum-of-squares (SOS) optimization~\cite{parrilo2000thesis,russ2010lqr,ani2013alongtraj} and the CBF for ACC is constructed by physics-based modeling and optimization. Through a QP that unifies the CBF and the performance controller, the designed controller mediates the safety and performance and is correct by construction under a clearly delineated set of assumptions. The practicality of the approach will be illustrated through calculations and simulation on a model of a mid-size passenger vehicle.

The rest of the paper is organized as follows. Section \ref{sec:cbf} introduces the definition and some theoretic results about control barrier functions. Section \ref{sec:preliminary} gives the dynamical models and the safety specifications for LK and ACC, based on which the composition problem studied in the paper is formulated and the assumptions and guarantees between LK and ACC are given. Section~\ref{sec:barrsyn} constructs CBFs for LK and ACC, respectively. Section \ref{sec:composition} gives the QP-based framework, based on which the input that solves the composition problem is provided. Simulation results are presented in Section~\ref{sec:simu} and finally, some conclusions in Section \ref{sec:conclusion}.

\emph{Notation}. The boundary and the interior of a set $\mathcal{S}$ are denoted as $\partial \mathcal{S}$ and $\mathrm{Int}(\mathcal{S})$ (or $\mathring{S}$), respectively. 
The commutative ring of real valued polynomials in $n$ variables $x_1,...,x_n$ is denoted as $\mathcal{R}[x_1,...,x_n]$, and as $\mathcal{R}_m[x_1,...,x_n]$ if its degree is $m$.
The set of sum of squares polynomials in $n$ variables $x_1,...,x_n$ is denoted as $\Sigma[x_1,...,x_n]$, and  as $\Sigma_m[x_1,...,x_n]$ if its degree is $m$.


%
%







\section{Preliminaries on Control Barrier Functions}\label{sec:cbf}





This section introduces the background on control barrier functions that will be used later.

Consider a nonlinear system on $ \R^n$,
\begin{eqnarray}
	\label{eqn:dynamicalsystem}
	\dot{x} = f(x),
\end{eqnarray}
with $f$ locally Lipschitz continuous. The solution of~\eqref{eqn:dynamicalsystem} with initial condition $x_0\in\mathbb{R}^n$ is denoted by $x(t,x_0)$ (or simply $x(t)$).
A set $\mathcal{S}$ is called {\it forward invariant} if for every $x_0 \in \mathcal{S}$, $x(t,x_0) \in \mathcal{S}$ for all $t \in I(x_0)$, where $I(x_0)$ is the \emph{maximal interval of existence} of $x(t,x_0)$.

Given a continuously differentiable function $h: \R^n \to \R$, define a closed set $\C$ as follows
\begin{align}
	\C &= \{ x \in \R^n : h(x) \geq 0\}. \label{eqn:superlevelsetC}
\end{align}
In what follows, it will also be assumed that $\C$ is nonempty and has no isolated points, that is, $\mathrm{Int}({\C}) \not= \emptyset$ and $\overline{\mathrm{Int}(\mathcal{C})} = \mathcal{C}$. 
{\color{black}The Lie derivative of $h(x)$ along the vector field $f(x)$ is denoted as $L_fh(x)$, that is, $L_fh(x)=\frac{\partial h}{\partial x}f(x)$.}

\begin{definition}
	\label{def:barrierfunctions2}\cite{Xu2015ADHS}
	Consider a  dynamical system \eqref{eqn:dynamicalsystem} and the set $\C$ defined by \eqref{eqn:superlevelsetC} for some continuously differentiable function $h: \R^n\rightarrow \R$. If there exist a constant $\gamma>0$ and a set $\D$ with $\C\subseteq\mathcal{D}\subset \R^n$ such that
	\begin{align}
		L_f h(x)& \geq  -\gamma h(x),\forall \; x \in \D,\label{eqn:generalinequality3}
	\end{align}
	then the function $h$ is called a (zeroing) barrier function.
\end{definition}

Existence of a (zeroing) barrier function implies the forward invariance of $\mathcal{C}$, as shown by the following theorem.

\begin{theorem}\label{thm:GBF}\cite{Xu2015ADHS}
	Given a dynamical system \eqref{eqn:dynamicalsystem} and a set $\C$ defined by \eqref{eqn:superlevelsetC} for some continuously differentiable function $h: \R^n\rightarrow \R$,
	if $h$ is a barrier function defined on the set $\D$ with $\C\subseteq\mathcal{D}\subset \R^n$, then $\C$ is forward invariant.
\end{theorem}

%

Consider an affine control system of the following form
\begin{eqnarray}
	\label{eqn:controlsys}
	\dot{x} = f(x) + g(x) u,
\end{eqnarray}
with $f$ and $g$ locally Lipschitz continuous, $x \in \R^n$ and $u \in U \subset \R^m$.

\begin{definition}\label{dfn:newcbf}\cite{Xu2015ADHS}
	Given a set $\mathcal{C} \subset \R^n$ defined by \eqref{eqn:superlevelsetC} for a continuously differentiable function $h: \R^n \to \R$,  {\color{black}the function $h$ is called} a (zeroing) control barrier function defined on   set $\mathcal{D}$ with $\C\subseteq\mathcal{D}\subset \R^n$, if  there exists a constant $\gamma>0$ such that\footnote{A more general definition for the (zeroing) CBFs that involves extended class $\mathcal{K}$ functions can be found in \cite{Xu2015ADHS}.}
	\begin{align}\label{ineq:ZCBF}
		& \sup_{u \in U}  \left[ L_f h(x) + L_g h(x) u + \gamma h(x)\right] \geq 0,\;\forall x \in \mathcal{D}.
	\end{align}
\end{definition}


Given a CBF $h$, for all $x\in\D$, define the set
\begin{equation}\label{zcbfinputset}
	\Kzcbf(x) =  \{ u \in U : L_f h(x) + L_g h(x) u + \gamma h(x) \geq 0\}.
\end{equation}
The following result guarantees the forward invariance of $\mathcal{C}$ when inputs are selected from $\Kzcbf(x)$.

\begin{theorem}\label{cor:zbf}\cite{Xu2015ADHS}
	Assume given a set $\mathcal{C} \subset \R^n$ defined by \eqref{eqn:superlevelsetC} for a continuously differentiable function $h$. If $h$ is a CBF on $\D$, then any locally Lipschitz continuous controller $u: \mathcal{D} \to U$ such that $u(x) \in \Kzcbf(x)$ will render the set $\mathcal{C}$ forward invariant.
\end{theorem}



{\color{black} In some cases,  seeking a CBF that is everywhere continuously differentiable can be too restrictive.
	Below, it is briefly pointed out how the assumption of continuous differentiability can be relaxed to a continuous function constructed from a finite set of continuously differentiable functions.}


{\color{black}
	Consider $p~(p\ge 2)$ continuously differentiable functions $h_1,...,h_p$ where $h_i: \R^n\rightarrow \R$. Assume that the sets
	$\C_i := \{ x \in \R^n : h_i(x) \geq 0\}$ satisfy
	$\mathrm{Int}({\C_i}) \not= \emptyset$ and $\overline{\mathrm{Int}(\mathcal{C}_i)} = \mathcal{C}_i$.
	Suppose that $\R^n$ is partitioned into $p$ closed sets $\mathcal{S}_1,...,\mathcal{S}_p$ such that $\cup_{i=1}^p \mathcal{S}_i =\R^n$ and $\mathrm{Int}(\mathcal{S}_i)\cap \mathrm{Int}(\mathcal{S}_j)=\emptyset$, $\forall i,j,i\neq j$.
	For any $i,j$ such that
	$\mathcal{S}_i \cap \mathcal{S}_j \neq \emptyset$, assume that $h_{i}(x)=h_{j}(x)$ for $x\in \mathcal{S}_i \cap \mathcal{S}_j $. Then the function $h:\R^n\rightarrow \R$ by
	\begin{align}\label{glueh}
		h|_{\mathcal{S}_i}(x)=h_i(x)
	\end{align}
	is well defined and continuous, and the set $\C=\{x\in\R^n|h(x)\ge 0\}$ is closed, has non-empty interior, and does not have isolated points. }

\begin{figure}[!htb]
	\centering
	\includegraphics[width=6.3cm]{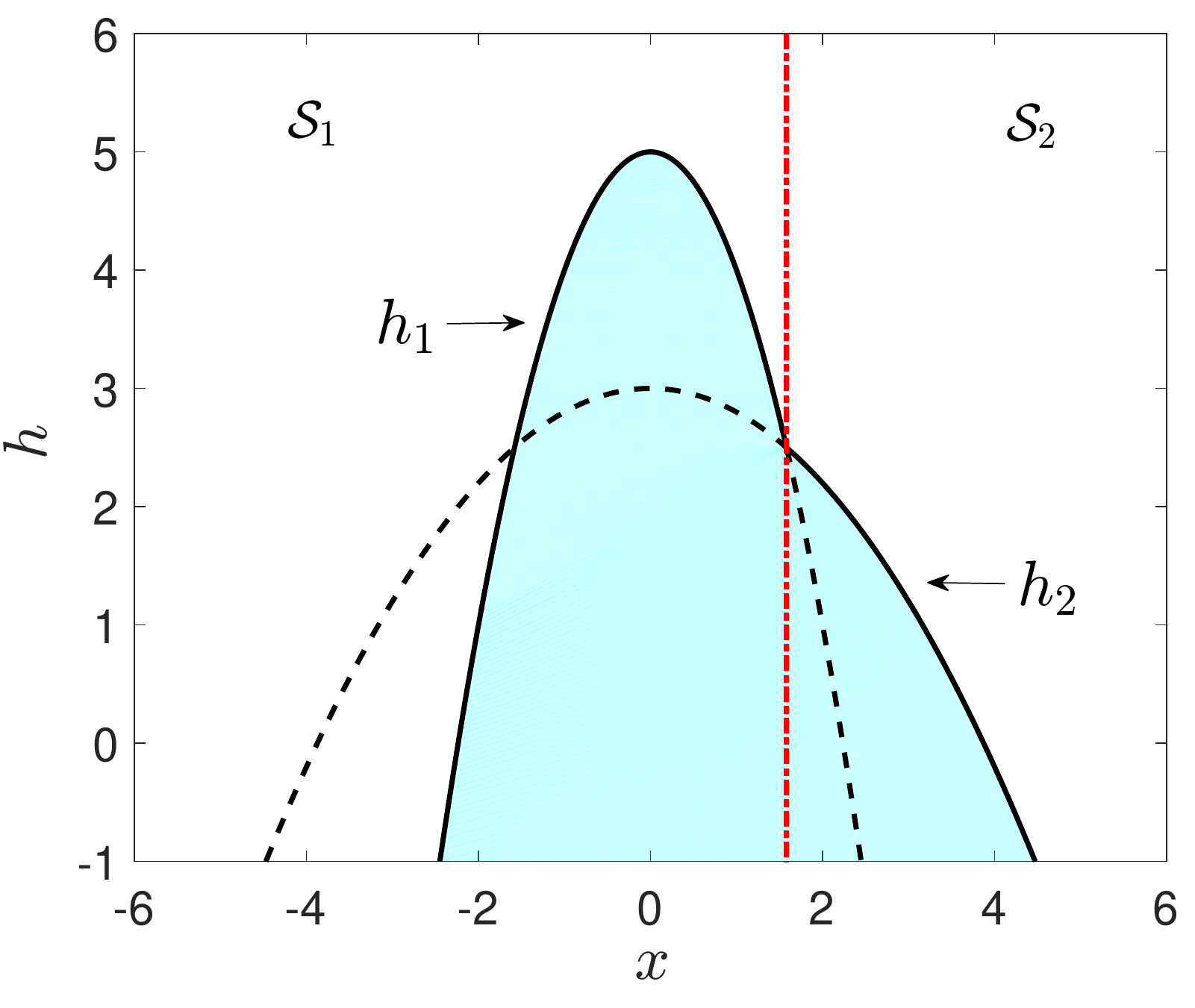}
	\caption{The safe set $\mathcal{C}$ defined in Example \ref{ex1} is the shaded region.}\label{fig:bar}
\end{figure}

{\color{black}
	\begin{example}\label{ex1}
		Consider two functions $h_1(x)=-x^2+5$, $h_2(x)=-0.2x^2+3$ and the partition $\mathcal{S}_1=(-\infty,\sqrt{10}/2]$, $\mathcal{S}_2=[\sqrt{10}/2,\infty)$. Clearly, $\mathrm{Int}(\mathcal{S}_1)\cap \mathrm{Int}(\mathcal{S}_2)=\emptyset$, $\mathcal{S}_1\cup \mathcal{S}_2=\R$ and $h_1(x)=h_2(x)$ when $x=\sqrt{10}/2$. Then, the safe set $\mathcal{C}:=\{x|h(x)\geq 0\}$ with $h(x)$ defined in \eqref{glueh} is shown as the shaded area in Figure \ref{fig:bar}.
	\end{example}}

	For all $x \in \mathcal{S}_i~(1\le i\le p)$, suppose that there exists a constant $\gamma_i>0$ such that
	\begin{align*}
		& \sup_{u \in U}  \left[ L_f h_i(x) + L_g h_i(x) u + \gamma_i h_i(x)\right] \geq 0,
	\end{align*}
	and define the set $K^i(x)$ as
	\begin{equation}
		K^i(x) =  \{ u \in U| L_f h_i(x) + L_g h_i(x) u + \gamma_i h_i(x) \geq 0\}. \nonumber
	\end{equation}
	
	For any $x\in\R^n$, define the set $K(x)$ as
	\begin{align}
		K(x)= \cap_{s\in\{i|x \in S_i\}} K^{s}(x).\label{con:glubar}
	\end{align}
	
	
	The following proposition extends Theorem \ref{cor:zbf} to the case of a safe set defined by a continuous function. Its proof is similar to that of Theorem \ref{cor:zbf} and is omitted here.
	
	
	\begin{proposition}\label{thm:gluebarrier}
		Given $p~(p\ge 2)$ continuously differentiable functions $h_1,...,h_p$ where $h_i: \R^n\rightarrow \R$ and a partition of $\R^n$ by $p$ closed sets $\mathcal{S}_1,...,\mathcal{S}_p$ such that
		if $\mathcal{S}_i \cap \mathcal{S}_j \neq \emptyset$, then $h_{i}(x)=h_{j}(x)$ for $x\in \mathcal{S}_i \cap \mathcal{S}_j $, define a function $h$ as in \eqref{glueh} and the set $\mathcal{C}=\{x\in\R^n|h(x)\ge 0\}$.
		If $K(x)\neq \emptyset$ for each  $x\in\R^n$, then any locally Lipschitz continuous controller $u: \R^n \to K(x)$ will render the set $\mathcal{C}$ forward invariant under the closed-loop system associated with \eqref{eqn:controlsys}.
	\end{proposition}
	
	

\section{Problem Formulation}\label{sec:preliminary}
In this section, we first introduce the individual models and specifications for LK and ACC, respectively. Then,  we formulate a composition problem that is studied in the paper, and propose an assume-guarantee contract between LK and ACC.



\subsection{Dynamic Models}

\begin{figure*}[!hbt]
	\begin{align}
		\left(  \begin{array}{c} \dot{y} \\ \color{black}{\dot{\nu}} \\ \dot{\Delta \psi}\\ \color{black}{\dot{r}} \end{array} \right)
		&= \left(  \begin{array}{cccc} 0 & 1 & \color{black}{v_f} & 0 \\
			0 & -\frac{ C_{f} + C_{r} }{ m\color{black}{v_f} } & 0 & \frac{b C_{r} - a C_{f} }{ m\color{black}{v_f} } - \color{black}{v_f} \\
			0 & 0 & 0 & 1  \\
			0 & \frac{b C_{r} - a C_{f} }{ I_{z} \color{black}{v_f} } & 0 & -\frac{a^2 C_{f} + b^2 C_{r} }{ I_{z} \color{black}{v_f} }
		\end{array} \right)
		\left( \begin{array}{c} y \\ \nu \\ \Delta \psi \\ r \end{array} \right)
		+ \left(  \begin{array}{c} 0 \\ \frac{C_{f} }{m} \\ 0 \\ \frac{aC_{f} }{I_{z}} \end{array} \right)u_1 + \left(  \begin{array}{r} 0 \\ 0 \\ -1 \\ 0 \end{array} \right)d. \label{LKmodel}
	\end{align}
	\begin{align}
		\label{ACCmodel}
		\left(  \begin{array}{c} \color{black}{\dot{v}_f} \\ \dot{v}_l \\ \dot{D} \end{array} \right)
		&= \left( \begin{array}{c} -\frac{c_0+c_1v_f}{m} \\ a_L \\ v_l-v_f \end{array} \right) +
		\left(  \begin{array}{c} \frac{1}{m} \\ 0 \\ 0 \end{array} \right)u_2 +
		\left(  \begin{array}{c} -\color{black}{\nu r} \\ 0 \\ 0 \end{array} \right).
	\end{align}
\end{figure*}

The LK model used in this paper is the lateral-yaw model as described in \cite{GerdLanePotential06,talvala2011pushing,rajamani2011bookvehicle}, and the ACC model is the point-mass model in \cite{petter2016cst,Aronbarriercdc14}. Their dynamics are given in \eqref{LKmodel} and \eqref{ACCmodel}, respectively.

Equation \eqref{LKmodel} describes the lateral-yaw dynamics of the controlled vehicle. In \eqref{LKmodel}, $\x_1:=(y,\nu,\Delta \psi,r)'$ is the state, where $y,\nu,\Delta \psi$ and $r$ represent the lateral displacement from the center of the lane, the lateral velocity, the yaw angle deviation in road-fixed coordinates, and the yaw rate, respectively; the input $u_1=\delta_f$ is the steering angle of the front wheels;  and $d$ is the desired yaw rate, which is viewed as a time-varying external disturbance and computed from road curvature by $d=v_f/R_0$ where $R_0$ is the (signed) radius of the road curvature and $v_f$ is the vehicle's longitudinal velocity. Moreover,
$m$ is the total mass of the vehicle, and $C_f,C_r,a$ and $b$ are parameters of the tires and vehicle geometry that are all positive numbers.

Equation \eqref{ACCmodel} describes the  longitudinal dynamics of the preceding vehicle and controlled vehicle. In \eqref{ACCmodel}, $\x_2:=(v_f,v_l,D)'$ is the state, which represent the following car's speed, the lead car's speed and the distance between them, respectively; $u_2=F_w$ is the input that represents the longitudinal force developed by the wheels; $F_r=c_0 + c_1 v_f +c_1 v_f^2$ is the aerodynamic drag, with constants $c_0,c_1,c_2$ that can be determined empirically; $a_L$ is the overall acceleration/deceleration of the lead car.

Equations \eqref{LKmodel} and \eqref{ACCmodel} are rewritten compactly as follows:
\begin{align}
	\dot{\x}_1 & = f_1(\x_1,v_f) + g_1(\x_1)u_1 + \Delta f_1(d), \label{PVLK}\\
	\dot{\x}_2 & = f_2(\x_2) + g_2(\x_2)u_2 + \Delta f_2(\nu r, a_L), \label{PVACC}
\end{align}
where
\begin{align*}
	& f_1(\x_1,v_f)=A_1(v_f)\x_1, \; g_1(\x_1)=B_1,\; \Delta f_1(d)=E_1d,\\
	&f_2(\x_2)=A_2\x_2,\;g_2(\x_2)=B_2,\;\Delta f_2(\nu r, a_L)=E_2,
\end{align*}
with
\begin{align*}
	A_1(v_f)&=\left(  \begin{array}{cccc} 0 & 1 & \color{black}{v_f} & 0 \\
		0 & -\frac{ C_{f} + C_{r} }{ m\color{black}{v_f} } & 0 & \frac{b C_{r} - a C_{f} }{ m\color{black}{v_f} } - \color{black}{v_f} \\
		0 & 0 & 0 & 1  \\
		0 & \frac{b C_{r} - a C_{f} }{ I_{z} \color{black}{v_f} } & 0 & -\frac{a^2 C_{f} + b^2 C_{r} }{ I_{z} \color{black}{v_f} }
	\end{array} \right), \\
	B_1 &= \left(  \begin{array}{c} 0 \\ \frac{C_{f} }{m} \\ 0 \\ \frac{aC_{f} }{I_{z}} \end{array} \right), \quad E_1 = \left(  \begin{array}{c} 0 \\ 0 \\ -1 \\ 0 \end{array} \right),
\end{align*}
\begin{align*}
	A_2&=\left(  \begin{array}{ccc} 0 & 0 & 0 \\
		0 & 0& 0  \\
		-1 & 1 & 0
	\end{array} \right), B_2 = \left(  \begin{array}{c} \frac{1}{m} \\0 \\ 0 \end{array} \right),\\
	E_2 &= \left(  \begin{array}{c} -\nu r-\frac{F_r}{m}  \\ a_L\\ 0 \end{array} \right).
\end{align*}

It is supposed that a bound is imposed on the steering angle $\delta_f$ of the controlled vehicle, that is, the set of admissible inputs for LK is
\begin{align}
	U_{lk}:=[- \hat \delta_f,\hat \delta_f ],\label{angleconstraint}
\end{align}
where $\hat \delta_f >0$ is the maximum steering angle. At highway speeds, this number will be much smaller than the maximum turning angle of the vehicle, say two or three degrees, versus 35 degrees.

It is also supposed that the acceleration/deceleration of the controlled car is bounded, that is, the set of admissible inputs for ACC is
\begin{align}
	& U_{acc}:=[-a_fmg, a_f'mg],\label{forceconstraint}
\end{align}
where $a_f,a_f'>0$. Typical bounds for driver comfort would be two or three tenths of gravitational acceleration. {\color{black} Note that $-a_fg$ is normally much less than the maximal deceleration capability of the car. Additionally, it is supposed that the deceleration/acceleration of the lead car is bounded. That is, $a_L\in[-a_lg,a_l'g]$ for some $a_l,a_l'>0$, which can be also expressed as $a_L\in\D_{a_L}$ where
	\begin{align}
		\D_{a_L}:=\{a\in\R |(a +a_lg)(a_l'g-a)\ge 0\}.
	\end{align}		
	The lead and controlled vehicles may have different allowable deceleration capabilities, i.e., $a_l$ and $a_f$ need not be equal.
}


\subsection{Specifications}
\label{sec:specifications}

{\color{black} Specifications for LK and ACC are given in this subsection. Among these specifications,
	the safety specifications are ``hard constraints'' that must be satisfied for all time, while the  performance objectives are ``soft constraints'' that can be overridden when they are in conflict with safety.}


{\bf LK Specifications.} The primary safety constraint of LK is to keep the car within its lane. That is, the absolute value of the lateral displacement $y$ is less than some given constant $y_{m}$, which is related to the width of the lane and the car. Specifically, this specification is expressed as
\begin{align}
	\label{eqn:LK_constraint}
	|y|\leq y_{m},
\end{align}
where $y_{m}$ is a given positive real number.

In addition to the lateral displacement, the other state variables should also be bounded. These bounds are stated as the following specifications:
\begin{align}
	\label{eqn:LK_constraint2}
	|\nu|\leq \nu_{m},\;|\Delta\psi|\leq \Delta\psi_{m},\;|r|\leq r_{m},
\end{align}
where $\nu_{m}$, $\Delta\psi_{m}$, $r_{m}$ are given positive real numbers.

A soft constraint for LK is for the vehicle's yaw rate $r(t)$ to track $d(t)$, the yaw or turning rate of the road, which is expressed  as
\begin{align}\label{LK_soft}
	& \lim_{t\rightarrow \infty} r(t) - d(t) = 0.
\end{align}
Another optional soft constraint is for the lateral acceleration to be upper bounded by a number that respects driver comfort, e.g., $|\dot{\nu}|\le 0.25g$.

We define the set $\mathcal{X}_{LK}$ as
\begin{align}
	&\mathcal{X}_{LK}:=\{\x_1\in\R^4||y|\le y_{m},|\nu|\le \nu_{m},\nonumber\\
	&\hskip 30mm|\Delta\psi|\le \Delta\psi_{m}, |r|\le r_{m} \}.\label{DX}
\end{align}
In what follows, we assume that $|d|\leq d_{\max}$ for some given $d_{\max}>0$, that is, $d\in\D_d$ where
\begin{align}
	\D_{d}=\{d\in\R | d^2_{\max} - d^2 \ge 0 \}.\label{Dd}
\end{align}

{\bf ACC Specifications.} {\color{black}The primary constraint for ACC is that the controlled vehicle \emph{maintain a safe distance from the lead car}. There are numerous formulations of this safety concept such as Time Headway and Time to Collision. In this paper, we use the following hard constraint  \cite{Vogel}:}
\begin{align}
	\label{eqn:ACC_constraint}
	D\geq \tau_{d}v_f +D_0,
\end{align}
where $\tau_{d}$ is the desired time headway
and $D_0$ is the minimal distance between cars when they are fully stopped.

The soft constraint for ACC, which is the performance objective of  the controlled car, is to achieve a desired speed $v_d$ set by the driver. This specification can be expressed as
\begin{align}
	\label{eqn:ACCobj}
	\lim_{t \to \infty} v_f(t)-v_d=0,
\end{align}
for a given positive constant $v_d$. {\color{black} Clearly, if the lead car's speed is less than $v_d$, then \eqref{eqn:ACCobj} cannot be achieved. In our control formulation, this is automatically taken into account without the need for if-then-else statements defining various modes of operation.}



\subsection{Formulating the Composition Problem}\label{subsec:problem}
The correctness guarantee for the composition of LK and ACC is formulated as the following problem.

{\color{black}{\it Given LK model \eqref{PVLK} and ACC model \eqref{PVACC}, find  feedback controllers $u_1\in U_{lk}$ and $u_2\in U_{acc}$ such that for any $d\in\D_d$ and $a_L\in\D_{a_L}$, the hard constraints \eqref{eqn:LK_constraint}, \eqref{eqn:LK_constraint2} and \eqref{eqn:ACC_constraint} are always satisfied, and the soft constraints \eqref{LK_soft} and  \eqref{eqn:ACCobj} are achieved  when they are not
		in conflict with the hard constraints.} }


{\color{black}
	Although the safety specifications for LK and ACC are given separately, the dynamics of LK and ACC interact with each other because $A_1(v_f)$ depends on $v_f$  and $\Delta f_2(\nu r, a_L)$ depends on the product of lateral velocity and yaw rate through the term $\nu r$. Furthermore, the external inputs $d,a_L$ and the assumption that $u_1,u_2$ are bounded make the compositional problem harder to slove (see Figure \ref{fig:com}).
	\begin{figure}[!htb]
		\centering
		\includegraphics[width=9cm]{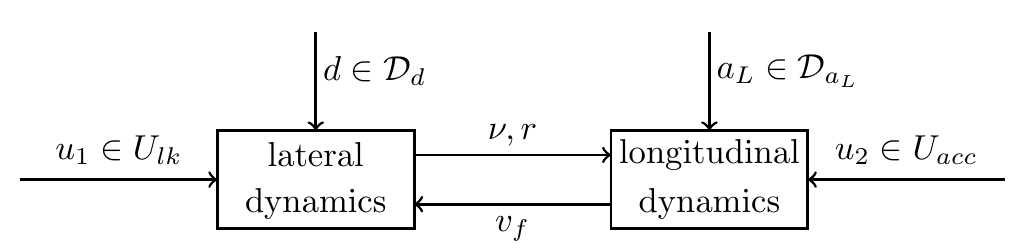}
		\caption{Interconnection of the lateral and longitudinal dynamics.}\label{fig:com}
	\end{figure}
}

\subsection{A Contract Between LK and ACC}\label{sub:contract}
{\color{black}One approach for the compositional design of complex systems is the \emph{contract-based method}, which formally defines a set of assume-guarantee protocols among subsystems in a ``circular'' manner~\cite{sangiovanni2012taming,benvenuti2008contract}.
	The basic idea is that each subsystem ensures that its behavior (i.e., set of trajectories) satisfies its own guarantees, under the assumption that all the other subsystems do the same, so that a formal guarantee on the behavior of the overall system can be established.
	For the composition of LK and ACC, we provide assumptions and guarantees between them in terms of the bounds of their respective coupling variables.

	Let $\bar v,\underline{v}>0$ be the upper and lower bound of $v_f$, respectively, where $\bar v\geq v_d$.
	Define a set $\D_{v_f}$ as
	\begin{align}
		\D_{v_f}&=\{v_f\in\R| (\bar{v} - v_f)(v_f - \underline{v})\ge 0\}.\label{Dvf}
	\end{align}
	The contract between the LK and the ACC subsystems is given as follows:
	\emph{The LK subsystem assumes that $v_f \in\D_{v_f}$ and guarantees that $\x_1\in\mathcal{X}_{LK}$ for any $d\in\mathcal{D}_d$, while the ACC subsystem assumes that $|\nu r|\le \nu_m r_m$ and guarantees that $v_f\in\D_{v_f}$ for any $a_L\in\mathcal{D}_{a_L}$.}

	The given contract allows us to design the controller for LK and ACC separately. As long as the individual specifications are respected in a formally correct way, the overall closed-loop system will also satisfy all the safety constraints in the compositional problem.
	In Section \ref{sec:barrsyn}, we will construct CBFs for LK and ACC individually, such that the contract will be respected provided that the two subsystems constrain their behaviors in a controlled invariant set corresponding to their individual CBF. In Section \ref{sec:composition}, we  will synthesize a provably correct solution to the composition problem via a QP that unifies CBFs and the performance controllers.}

\section{Control Barrier Functions For LK and ACC}\label{sec:barrsyn}




In this section, we provide CBFs for the LK model \eqref{LKmodel} and the ACC model \eqref{ACCmodel}, respectively.
To this end, we translate the hard constraints of Section~\ref{sec:specifications} into conditions that a CBF must satisfy so that, by the results of Section~\ref{sec:cbf}, the trajectories of the closed-loop system will satisfy the safety portion of the specification.

\subsection{The CBF For LK}


For the LK subsystem, we construct a polynomial function $h_{lk}(\x_1)\in \mathcal{R}_{\alpha}[\x_1]$  that has the following form:
\begin{align}\label{hlktemp}
	h_{lk}(\x_1) = \kappa - \hat h_{lk}(\x_1)
\end{align}
where $\alpha$ is some given positive integer indicating the polynomial degree, $\kappa\in\R$ is some positive number, $\hat h_{lk}(\x_1)\in \mathcal{R}_{\alpha}[\x_1]$ is a polynomial that is nonnegative. Then the safe set for LK is defined as
$$
\C_{lk}:=\{\x_1\in\R^4|h_{lk}(\x_1)\ge 0\}.
$$
Here, $\kappa$ is a variable used to enlarge the volume of $\C_{lk}$, which will be discussed in detail later.


According to the contract in Subsection \ref{sub:contract}, the CBF $h_{lk}(\x_1)$ will be designed to satisfy the following properties:
\begin{align}
	&\mbox{(LK-P1)}\;\C_{lk}:=\{\x_1\in\R^4|h_{lk}(\x_1)\ge 0\} \neq \emptyset,\label{lkcon3}\\
	&\mbox{(LK-P2)}\; \C_{lk}\subset \mathring{\mathcal{X}}_{LK}, \label{lkcon1}\\
	&\mbox{(LK-P3)}\;\forall \x_1\in\C_{lk}, \forall v_f\in\D_{v_f},\forall d\in\D_d,\nonumber\\
	&\; \sup_{u_1 \in U_{lk}}  \left[ L_{f_1+\Delta f_1} h_{lk}(\x_1) + L_{g_1} h_{lk}(\x_1) u_1 +\gamma h_{lk}(\x_1)\right] \geq 0,\label{lkcon2}
\end{align}
where $\gamma>0$ is a given number representing the gain in  \eqref{ineq:ZCBF}.

The property (LK-P1) ensures that the safe set $\C_{lk}$ is non-empty, property (LK-P2) ensures that $\x_1\in \mathcal{X}_{LK}$ as long as $\x_1\in \C_{lk}$, and property (LK-P3) ensures that $h_{lk}$ is a control barrier function for any longitudinal velocities in $\D_{v_f}$ and any desired yaw rates in $\D_d$. Note that properties (LK-P1)-(LK-P3) imply the satisfaction of the hard constraints \eqref{eqn:LK_constraint} and \eqref{eqn:LK_constraint2}.

In summary, $h_{lk}$ with properties (LK-P1)-(LK-P3) guarantees that, for any speed $v_f\in\D_{v_f}$ and desired yaw rate $d\in\D_d$, there exists steering angle $\delta_f \in U_{lk}$ such that the trajectory of the LK system stays in the set $\C_{lk}$ (and thus set $\mathcal{X}_{LK}$) if starting from $\C_{lk}$.
%


{\color{black}Recalling that  the set $\D$ is the region for which the CBF condition holds (cf. Definition \ref{def:barrierfunctions2}), Theorem \ref{thmLK} provides a sufficient condition for the existence of $h_{lk}$ with $\D=\mathcal{X}_{LK}$, which is the key step in translating the properties (LK-P1)-(LK-P3) into a set of sufficient conditions that can then be synthesized by sum-of-squares programs~\cite{parrilo2000thesis,anderson2013robustF18}.
	In fact, we assume that $\D=\mathcal{X}_{LK}$ for now, which simplifies the procedure to construct $h_{lk}$, and we will discuss the case  $\D=\C_{lk}$ later.}


\begin{theorem}\label{thmLK}
	Given the LK model \eqref{PVLK}, the variable bounds  $y_m,\nu_m,\Delta\psi_m,r_m$, the admissible set $U_{lk}$ defined in \eqref{angleconstraint}, a positive definite polynomial $p(\x_1)$ and the sets $\D_{v_f},\D_d$ defined in \eqref{Dvf} and \eqref{Dd}, if there exist $\gamma>0$, $\rho>0$, some positive integer $\alpha$, polynomial $h_{lk}(\x_1)\in\mathcal{R}_{\alpha}[\x_1]$, non-negative polynomials  $s_0,s_1,...,s_4\in\Sigma[\x_1]$ and $s_5,...,s_{10}\in\Sigma[\x_1,d,v_f]$ such that
	\begin{align}
		&h_{lk}(\x_1) - (\rho - p(\x_1))s_0(\x_1)\ge 0,\label{thmcon1}\\
		&(y^2-y_m^2)s_{1} + h_{lk}(\x_1) < 0, \label{thmcon2}\\
		&(\nu^2-\nu_m^2)s_{2} + h_{lk}(\x_1) < 0, \label{thmcon3}\\
		&(\Delta\psi^2-\Delta\psi_m^2)s_{3} + h_{lk}(\x_1) < 0, \label{thmcon4}\\
		&(r^2-r_m^2)s_{4} + h_{lk}(\x_1) < 0, \label{thmcon5}\\
		&\sup_{u_1 \in U_{lk}}  \left[ L_{f_1+\Delta f_1} h_{lk}(\x_1) + L_{g_1} h_{lk}(\x_1) u_1 +\gamma h_{lk}(\x_1)\right] \nonumber\\
		&-(y_m^2-y^2)s_{5}  -(\nu_m^2-\nu^2)s_{6}- (\Delta \psi_m^2-\Delta \psi^2)s_{7} \nonumber\\
		&- (r^m_2-r^2)s_{8} - (d_{\max}^2-d^2)s_{9}- (\bar v-v_f)(v_f-\underline{v})s_{10}  \ge 0,\label{thmcon6}
	\end{align}
	then $h_{lk}(\x_1)$ satisfies properties (LK-P1)-(LK-P3) defined in \eqref{lkcon3}-\eqref{lkcon2}.
\end{theorem}
\begin{proof}
	Condition \eqref{thmcon1} implies (LK-P1) because $h_{lk}(\x_1)\ge 0$ whenever $p(\x_1)\le \rho$, which means that
	$\{\x_1|p(\x_1)\le \rho\}\subseteq \C_{lk}$ and therefore $\C_{lk}\neq \emptyset$.
	Based on the S-procedure, conditions \eqref{thmcon2}-\eqref{thmcon5} imply (LK-P2) because $|y|< y_m$, $|\nu|< \nu_m$, $|\Delta\psi|< \Delta\psi_m$ and $|r|< r_m$ whenever $h_{lk}(\x_1)\ge 0$. Condition \eqref{thmcon6} implies condition (LK-P3), since \eqref{lkcon2} holds whenever $\x_1\in\mathcal{X}_{LK},v_f\in\D_{v_f},d\in\D_d$.
\end{proof}



In addition to the properties (LK-P1)-(LK-P3), it is desirable for $h_{lk}$ to maximize the volume of $\C_{lk}$, the portion of the safe set that the CBF renders controlled invariant. To this end, we normalize the coefficients of $\hat h_{lk}$ by adding the constraint $\hat h_{lk}((1,...,1)^\top)=1$ (i.e., the sum of the coefficients is equal to 1) and maximizing $\kappa$. Note that the normalization is necessary to make the maximization of  $\kappa$ valid, since otherwise the coefficients of $\hat h_{lk}$ can be scaled accordingly with $\kappa$ that results in the same $\C_{lk}$, which makes maximizing $\kappa$ have no meaning.
The normalization method has also been used in finding the maximal region of attraction of a control system using Lyapunov function and SOS \cite{ani2013alongtraj,tanphdthesis}.

The terms $\rho s_0$ in \eqref{thmcon1} and $\gamma h_{lk}$ in \eqref{thmcon6} involve products of the unknowns, but it can be turned into linear constraints on the unknowns through bisecting $\rho$ and $\gamma$. Moreover, the maximization over the allowed steering angles in  \eqref{thmcon6}  also involves a product of the unknowns, which can be overcome by finding an explicit formula for the controller and iterating between improving the current control solution and the CBF.

In what follows, we explain in detail the iterative procedure to construct $h_{lk}$, which satisfies \eqref{thmcon1}-\eqref{thmcon6} and maximizes the volume of  $\C_{lk}$,  using SOS programs.




\emph{1.Initialization.}\; {\color{black}
	The iteration process is initialized by first computing an LQR controller for a nominal value of $v_f\in \D_{v_f}$ and $d=0$. More specifically, we first choose weight matrices $Q\in\R^{4\times 4}$ penalizing the state and $R\in\R$ penalizing the input, where look-ahead of the road curvature can be taken into account in $Q$~\cite{GerdLanePotential06}. Then we solve for the LQR gain $K\in\R^{1\times 4}$ for the system $(A(v_f),B)$ with a given and fixed $v_f\in \D_{v_f}$.} Based on the gain $K$, we fix the control $u_1=\frac{-K\x_1}{1+\eta(K\x_1)^2}$ and solve for $h_{lk}$ using a feasibility SOS program, with $\eta$ selected as $\eta=1/(2\hat \delta_f)^2$. This value implies that $\lvert\frac{-K\x_1}{1+\eta(K\x_1)^2}\rvert\le \hat \delta_f$ and therefore $u_1\in U_{lk}$.

\begin{remark}
	The motivation for scaling the LQR control $-K\x_1$ by $1+\eta(K\x_1)^2$ includes: (i) an LQR controller gain $K$ is easily computed; (ii) {\color{black}the value function of the LQR controller corresponds to a quadratic form}, and  there always exists a sufficiently small sublevel set of the quadratic form contained in ${\mathcal{X}}_{LK}$, for which the values of the controller over that sublevel set lie in $U_{lk}$; and (iii), while there is no guarantee that the controller $u_1$ computed in this manner will result in a feasible SOS program (i.e., the following $(\mathcal{P}_0)$) for all $v_f\in \D_{v_f}$ and $d\in \D_{d}$, this has not been a problem in the examples we have worked when $\rho$ is sufficiently small. Of course, the user is free to use other control synthesis methods to initiate the iteration process. It is also a choice whether or not to feed forward the desired yaw rate.
\end{remark}

Given an LQR gain $K$, we choose values for $\gamma>0$ and sufficiently small $\rho_0>0,\epsilon>0$ and use the following feasibility SOS program for initialization.



\begin{align}
	&(\mathcal{P}_0):\nonumber\\
	&\mbox{find}\;h_{lk}\in\mathcal{R}_{\alpha}[\x_1],s_0,s_1,...,s_{4} \in\Sigma[\x_1],s_5,...,s_{10}\in\Sigma[\x_1,d,v_f] \nonumber\\
	&\mbox{such that}\nonumber\\
	&h_{lk} - (\rho_0 - p)s_0\in\Sigma[\x_1], \label{c0}\\
	&-h_{lk}-(y^2-y_m^2)s_{1}-\ep\in\Sigma[\x_1],\label{c1}\\
	&-h_{lk}-(\nu^2-\nu_m^2)s_{2}-\ep\in\Sigma[\x_1],\label{c2}\\
	&-h_{lk}-(\Delta \psi^2-\Delta \psi_m^2)s_{3}-\ep\in\Sigma[\x_1],\label{c3}\\
	&-h_{lk}-(r^2-r_m^2)s_{4}-\ep\in\Sigma[\x_1],\label{c4}\\
	&\frac{\partial h_{lk}}{\partial \x_1}[A(v_f)\x_1+Ed][1+\eta(K\x_1)^2]v_f+\frac{\partial h_{lk}}{\partial \x_1}B(-K\x_1)v_f\nonumber\\
	&\quad+\gamma h_{lk}[1+\eta(K\x_1)^2]v_f - (y_m^2-y^2)s_{5} -(\nu_m^2-\nu^2)s_{6}\nonumber\\
	&\quad- (\Delta \psi_m^2-\Delta \psi^2)s_{7} - (r^2_m-r^2)s_{8} - (d_{\max}^2-d^2)s_{9} \nonumber\\
	&\quad- (\bar v-v_f)(v_f-\underline{v})s_{10} \in\Sigma[\x_1,d,v_f],\label{c5}
\end{align}
where $\eta=1/(2\hat \delta_f)^2$.

%


Note that \eqref{c0} implies (LK-P1), \eqref{c1}-\eqref{c4} imply (LK-P2) based on the S-procedure, and  \eqref{c5} implies that for any $\x_1\in\mathcal{X}_{LK}$, $v_f\in\D_{v_f}$,  $d\in\D_d$,
\begin{align}
	\frac{\partial h_{lk}}{\partial \x_1}[A(v_f)\x_1+B\frac{-K\x_1}{1+\eta(K\x_1)^2}+Ed]+\gamma h_{lk}\ge 0.\label{LKini}
\end{align}

Thus, \eqref{LKini} means that the control $u_1=\frac{-K\x_1}{1+\eta(K\x_1)^2}\in U_{lk}$ results in the CBF condition $\dot{h}_{lk}(\x_1,v_f,d,u_1)+\gamma h_{lk}(\x_1)\ge 0$ holding.
Because $A(v_f)$ is a rational matrix with the denominator $v_f$ in some of its entries, it needs to be multiplied with $v_f$ so that it becomes a polynomial \cite{anderson2013robustF18}. Note that in $(\mathcal{P}_0)$ and in what follows, the degrees of the multipliers $s_i$ are not specified explicitly and assumed to be chosen appropriately.


If $(\mathcal{P}_0)$ is infeasible, we repeat it by modifying parameters $Q,R,\gamma,\rho_0,\epsilon$ and increasing $\alpha$; otherwise, $h_{lk}$ is obtained as a polynomial that satisfies properties (LK-P1)-(LK-P3). Then, the following two steps will be used to find a polynomial $h_{lk}$ that increases the volume of $\C_{lk}$.




\emph{2.Synthesize Controller.}\;Given the polynomial $h_{lk}$ from the initialization step, which is denoted as $h_{lk}^{old}$ in the subsequent $(\mathcal{P}_1)$, we use the following maximization SOS program to find a new controller $u\in\mathcal{R}_{\beta}[\x_1,d,v_f]$  with some positive integer $\beta$, $\kappa\in\R$ and multipliers $s_i(0\le i\le 22)$, such that  $u\in U_{lk}$, $\kappa$ is maximized and for any $\x_1\in\mathcal{X}_{LK},v_f\in\D_{v_f},d\in\D_d$, the CBF condition \eqref{lkcon2} holds.

\begin{align}
	&(\mathcal{P}_1):\nonumber\\
	&\max \;\kappa\;\nonumber\\
	&\;\mbox{over}\;\kappa\in\R,u\in\mathcal{R}_{\beta}[\x_1,d,v_f],s_{0},s_1,...,s_{4} \in\Sigma[\x_1],\nonumber\\
	&\;s_5,...,s_{22}\in\Sigma[\x_1,d,v_f],\;\mbox{such that}\nonumber\\
	& h_{lk} - h_{lk}^{old}s_0\in\Sigma[\x_1], \label{cc0}\\
	&\;\eqref{c1}-\eqref{c4}\;\mbox{hold},\nonumber\\
	&\frac{\partial h_{lk}}{\partial \x_1}[A(v_f)\x_1+Bu(\x_1,d,v_f) +Ed]v_f+\gamma h_{lk}v_f  \nonumber\\
	& - (y_m^2-y^2)s_{5} -(\nu_m^2-\nu^2)s_{6} - (\Delta \psi_m^2-\Delta \psi^2)s_{7} \nonumber\\
	& - (r^2_m-r^2)s_{8}  - (\bar v-v_f)(v_f-\underline{v})s_{9}   \nonumber \\
	& - (d_{\max}^2-d^2)s_{10}  \in\Sigma[\x_1,d,v_f], \label{c8}\\
	& u(\x_1,d,v_f)+\hat\delta_f - (y_m^2-y^2)s_{11} - (\nu_m^2-\nu^2)s_{12}\nonumber\\
	& - (\Delta \psi_m^2-\Delta \psi^2)s_{13} - (r^2_m-r^2)s_{14} - (d_{\max}^2-d^2)s_{15} \nonumber\\
	&\quad- (\bar v-v_f)(v_f-\underline{v})s_{16} \in\Sigma[\x_1,d,v_f],\label{c6}\\
	& -u(\x_1,d,v_f)+\hat\delta_f - (y_m^2-y^2)s_{17} - (\nu_m^2-\nu^2)s_{18}\nonumber\\
	& - (\Delta \psi_m^2-\Delta \psi^2)s_{19} - (r^2_m-r^2)s_{20} - (d_{\max}^2-d^2)s_{21} \nonumber\\
	&\quad- (\bar v-v_f)(v_f-\underline{v})s_{22} \in\Sigma[\x_1,d,v_f].\label{c7}
\end{align}




Condition \eqref{c8} means that with such $u$, the CBF condition $\dot{h}_{lk}(\x_1,v_f,d,u)+\gamma h_{lk}(\x_1)\ge 0$ holds, for any $\x_1\in\mathcal{X}_{lk},d\in\D_d,v_f\in\D_{v_f}$.
Conditions \eqref{c6}-\eqref{c7} mean that the synthesized $u\in\mathcal{R}_{\beta}[\x_1,d,v_f]$ satisfies $|u|\le \hat \delta_f$, which implies that $u\in U_{lk}$, for any $\x_1\in\C_{lk},d\in\D_d,v_f\in\D_{v_f}$.


When the procedure goes from $(\mathcal{P}_0)$ to $(\mathcal{P}_1)$, which will happen only once, it is not guaranteed that $(\mathcal{P}_1)$ will be feasible.  In case of infeasibility,  we can increase the degree of $u$ and repeat $(\mathcal{P}_1)$. However, in all examples we have worked, $(\mathcal{P}_1)$ has been feasible, even when $u$ is chosen to be of degree two. On the other hand, $(\mathcal{P}_1)$ will always be feasible when executed after $(\mathcal{P}_2)$, another SOS program that will be discussed shortly. {\color{black} Therefore, we assume that $(\mathcal{P}_1)$ is feasible at initialization and proceed.}

\begin{remark}\label{rem:P1}
	If we choose the controller in $(\mathcal{P}_1)$ to be a rational function in the form of $\tilde K_1 \x_1/(1+\x_1^\top \tilde K_2 \x_1)$ where $\tilde K_1 \in \R^{1\times 4}$, $\tilde K_2\in \R^{4\times 4}$ positive definite, then $(\mathcal{P}_1)$ is guaranteed to be feasible because \eqref{LKini} holds and $u_1$ in $(\mathcal{P}_0)$ is a rational polynomial function of the same form. Higher order terms can also be included in the numerator/denominator of the rational function. The following algorithms remain true after appropriate modifications, if the rational functions template are used for the controller. This also validates the above assumption that $(\mathcal{P}_1)$ is always feasible. 
\end{remark}

\emph{3.Synthesize Barrier.}\; Given the controller $u(\x_1,d,v_f)$ and the CBF $h_{lk}$ from $(\mathcal{P}_1)$, which will be denoted as $h_{lk}^{old}$ in the subsequent $(\mathcal{P}_2)$, the following SOS program finds a new CBF $h_{lk}$ and multipliers $s_0,s_1,...,s_{10}$ to maximize $\kappa$.


\begin{align}
	&(\mathcal{P}_2):\nonumber\\
	&\max \;\kappa\;\nonumber\\
	&\;\mbox{over}\;\kappa\in\R,\hat h_{lk}\in\mathcal{R}_{\alpha}[\x_1],s_0,s_1,..., s_{4}\in\Sigma[\x_1],\nonumber\\
	&\quad \quad \;s_{5},...,s_{10}\in\Sigma[\x_1,d,v_f],\mbox{such that}\nonumber\\
	&\eqref{c1}-\eqref{c4},\eqref{cc0}\;\mbox{and}\;\eqref{c8}\;\mbox{hold}.\nonumber
\end{align}

Note that $(\mathcal{P}_2)$ is always feasible since $u$ and $\kappa$ in $(\mathcal{P}_1)$ constitute a feasible solution, and the resulting $h_{lk}$ is a polynomial satisfying \eqref{thmcon1}-\eqref{thmcon6} and therefore properties (LK-P1)-(LK-P3).

With the new CBF $h_{lk}$ constructed, we return to Step 2 to continue the iterative procedure until convergence. Because  $\mathcal{X}_{LK}$ is a compact set and the constructed set $\C_{lk}$ in each step of $(\mathcal{P}_1)$-$(\mathcal{P}_2)$ is no smaller than in the previous step, asymptotic convergence is guaranteed. In practice, we can either terminate the algorithm when the change of $\kappa$ is below some threshold, or simply set in advance the number of iterations. Furthermore, when we return to $(\mathcal{P}_1)$, it is guaranteed to be feasible since $u$ in the last step, i.e., $(\mathcal{P}_2)$, is a feasible solution.






Algorithm \ref{algo1} and Proposition \ref{prp:algoCBF} summarize the above results.


\begin{algorithm}[!hbt]
	\caption{Synthesis of Control Barrier Functions for LK}
	\begin{algorithmic}[1]\label{algo1}
		\REQUIRE $y_m,\nu_m,\Delta\psi_m,r_m,\hat\delta_f,d_{\max},\bar v,\underline{v},Q,R,\gamma,\ep,\rho_0,p,\alpha,\beta$
		\ENSURE $\kappa,\hat h_{lk}(\x_1),u(\x_1,d,v_f)$
		\STATE Solve for the LQR gain $K$ and solve $(\mathcal{P}_0)$
		\WHILE {$(\mathcal{P}_0)$ is not feasible}
		\STATE {Modify $Q,R,\gamma,\rho_0$ and solve $(\mathcal{P}_0)$  }
		\ENDWHILE
		\STATE converged = false
		\WHILE { $\neg$ converged}
		\STATE Fix $\hat h_{lk}$, find $u,s_i,\kappa$ and maximize $\kappa$ by solving $(\mathcal{P}_1)$
		\STATE	Fix $u$, find $\hat h_{lk},s_i,\kappa$ and maximize $\kappa$  by solving $(\mathcal{P}_2)$
		\IF {$|\kappa^{new}-\kappa^{old}|\le$ some threshold}
		\STATE converged = true
		\ENDIF
		\ENDWHILE
	\end{algorithmic}
\end{algorithm}


\begin{proposition}\label{prp:algoCBF}
	If $(\mathcal{P}_0)$ is feasible, then Algorithm \ref{algo1} terminates and the polynomial $h_{lk}(\x_1)$ returned by it satisfies properties (LK-P1)-(LK-P3). 
\end{proposition}

%




\begin{remark}
	There are no efficient and reliable solvers for semi-definite programs with bilinear constraints in the decision variables, which are non-convex and known to be NP-hard in general. Iterative procedures have therefore been commonly used to bypass the bilinear constraints for SOS programs; for instance, they were used to search for control Lyapunov functions in \cite{tan2004searching} and to construct an invariant funnel along trajectories in \cite{ani2013alongtraj}. However, in contrast to the cited results, the particular controller constructed in Algorithm \ref{algo1} is not important to us since it will not be implemented directly on the system; indeed, the actual control input will be generated by solving a quadratic program that will be explained in Section \ref{sec:composition}. In fact, it is the CBF $h_{lk}$ that is crucial to us, because it characterizes the safe set $\C_{lk}$ that can be rendered controlled invariant using input values selected from $U_{lk}$. These observations allow us to focus on the construction of the CBFs instead of the control law (recall the discussion in Remark \ref{rem:P1} about the flexible form of the controller).
\end{remark}





By fixing $h_{lk}$ and $u$ obtained from Algorithm \ref{algo1}, we can further maximize $\gamma$ by solving the following SOS program:
\begin{align}
	&\max \;\gamma\;\nonumber\\
	&\mbox{over}\;\gamma\in\R,s_0,..., s_{4}\in\Sigma[\x_1],s_{5},...,s_{10}\in\Sigma[\x_1,d,v_f]\nonumber\\
	&\mbox{such that}\;\eqref{c8}\;\mbox{holds}.\nonumber
\end{align}
With the maximal $\gamma$, we obtain the maximal allowable input set $\Kzcbf(x)$ (cf. Definition \ref{zcbfinputset}) w.r.t. the CBF $h_{lk}$, from which the input can render the set $\C_{lk}$ controlled invariant under the dynamics of the LK system.




%

\begin{remark}\label{rem:setD}
	Since the CBF  $h_{lk}$ satisfies $\dot{h}_{lk}+\gamma h_{lk}\ge 0$ in $\mathcal{X}_{LK}$,  the set $\C_{lk}$ is asymptotically stable in $\mathcal{X}_{LK}$ under a control law taking values from $\Kzcbf(x)$. Therefore, we can take into account affine disturbances in \eqref{LKmodel} similar to the argument in \cite{Xu2015ADHS}, {\color{black}by which it can be shown that the LK system is input-to-state stable with respect to the disturbances and a larger controlled invariant set containing $\C_{lk}$ can be quantitatively given.}
\end{remark}

The argument above shows that, if we choose $\D=\mathcal{X}_{LK}$, then $\C_{lk}$ is controlled invariant and is attractive within $\mathcal{X}_{LK}$ under control from $\Kzcbf(x)$. On the other hand, if we choose $\D=\C_{lk}$, we will get a larger set $\C_{lk}$ in principle (since it is not attractive outside $\C_{lk}$), but constructing $h_{lk}$ that defines such $\C_{lk}$ would then become more involved. Note that for this case, Theorem \ref{thmLK} is still true if condition \eqref{c8} is changed into:
\begin{align}
	&\frac{\partial h_{lk}}{\partial \x_1}[A(v_f)\x_1+Bu(\x_1,d,v_f) +Ed]v_f+\gamma h_{lk}v_f  - h_{lk}s_{5}   \nonumber\\
	&-(\bar v-v_f)(v_f-\underline{v})s_{6}-(d_{\max}^2-d^2)s_{7}\in\Sigma[\x_1,d,v_f].\label{DClk}
\end{align}
where $s_5,s_6,s_{7}\in\Sigma[\x_1,d,v_f]$ are multipliers to be found.

As shown in \eqref{DClk}, $h_{lk}s_{5}$ is an additional bilinear term of the unknowns if SOS programs are applied to construct $h_{lk}$ for $\D=\C_{lk}$. To bypass this difficulty, we can divide $(\mathcal{P}_1)$ into two steps as follows: $(i)$ fix $h_{lk}$, search for $u\in\mathcal{R}_{\beta}[\x_1,d,v_f],s_i\in\Sigma[\x_1,d,v_f]$ by solving a feasibility SOS program, $(ii)$ fix $\hat h_{lk}$, the control $u$ and the multipliers $s_i$ obtained in $(i)$, search for $\kappa$ and maximize it by solving a maximization SOS program. Then, if Line 7 of Algorithm \ref{algo1} is replaced with these two steps, the resulting CBF $h_{lk}$ will satisfy properties (LK-P1)-(LK-P3).

\begin{remark}
	A bound on lateral acceleration $\dot{\nu}$, which was introduced as a soft constraints for LK in Subsection \ref{sec:specifications}, can be added as a (hard) constraint to the SOS programs. However, by doing this, the feasibility of $(\mathcal{P}_1)$ and $(\mathcal{P}_2)$ will no longer be guaranteed. Thus, this constraint is not considered, and will be discussed later in Section \ref{sec:composition}.
\end{remark}

\begin{remark} Using the SOS optimization is not the only way to design CBFs. Gerdes et al. developed a Lagrangian model of the lateral dynamics and then augmented the corresponding Hamiltonian with an additional potential term to enforce the invariance of a set delineated by the lane boundaries, in the face of road curvature variations \cite{GerdLanePotential06}. However, with this method, it is unclear how to address bounds on steering angle, yaw rate and lateral velocity, as we have done in \eqref{thmcon2}-\eqref{thmcon5}; moreover, there is no distinction between safety--staying within the lane markers---and performance---how much to override the driver or how close to remain centered in the lane. When applying LQR to the LK problem, the cost-to-go function resulting from solving the Riccati equation for a constant longitudinal speed $v_f$ does yield a quadratic barrier function for the closed-loop lateral-yaw model, with $v_f$ in a small neighborhood of the nominal speed, and hence is also a CBF for the open-loop lateral-law model for the same range of longitudinal speed. However, the LQR approach to developing a CBF cannot handle the bounded input/state  or the varying road curvature constraints; moreover, our experience is that the associated safe set computed from a sub-level set of the cost-to-go function is unacceptably small.
	
	
\end{remark}

\subsection{The CBF For ACC}

{\color{black}Suppose that the CBF for the ACC subsystem has the following form:
	\begin{align}\label{accform}
		h_{acc}(\x_2):&=D-\tau_{d}v_f -D_0 - \hat h_{acc}(v_f,v_l)
	\end{align}
	where $\hat h_{acc}(v_f,v_l)$ is a polynomial to be determined. According to the contract in Section \ref{sub:contract}, the CBF $h_{acc}$ will be designed to satisfy the following properties:}
\begin{align}
	&\mbox{(ACC-P1)}\;\C_{acc}:=\{\x_2\in\R^3|h_{acc}(\x_2)\ge 0\} \neq \emptyset,\label{accon1}\\
	&\mbox{(ACC-P2)}\; \forall v_f\in \D_{v_f}, \forall v_l\geq \underline{v},  \hat h_{acc}(v_f,v_l)\ge 0, \label{accon2}\\
	&\mbox{(ACC-P3)}\;\forall \x_2\in \mathcal{X}_{ACC}, \forall a_L\in\D_{a_L},\forall |\nu r|\le \nu_m r_m,\nonumber\\
	&\quad\; \sup_{u_2 \in U_{acc}}  \left[ L_{f_2+\Delta f_2} h_{acc}(\x_2) + L_{g_2} h_{acc}(\x_2) u_2 \right] \geq 0.\label{accon3}
\end{align}

{\color{black}When $h_{acc}(\x_2)=0$, the minimal safe distance for the controlled car is given by $D_{\min}=\tau_{d}v_f +D_0 + \hat h_{acc}(v_f,v_l)$, which is no less than $\tau_{d}v_f +D_0$ since  $\hat h_{acc}(v_f,v_l)\geq 0$. Therefore, the hard constraint \eqref{eqn:ACC_constraint} will be satisfied when $\x_2\in \C_{acc}$. Note that condition \eqref{accon3} implies the controlled invariance of $\mathcal{C}_{acc}$.  We also point out that the CBF cannot be simply chosen as $h_{acc}(\x_2):=D-\tau_{d}v_f -D_0$, since the set $\C_{acc}$ thus defined is not controlled invariant using an input $u_2\in U_{acc}$.

	It is clear that an overly conservative safe distance $D_{\min}$ is undesirable, as it will encourage other cars to cut into the lane.
	When the SOS program is used to construct $h_{acc}(\x_2)$, however, the resulting $D_{\min}$ was unnecessarily large.
	A physics-based optimization can be used to construct $h_{acc}$ by noting that
	the ACC subsystem has the following monotone property: if $(D_1,v_f,v_l)\in\C_{acc}$ when $a_L=-a_lg$ and $u_2=-a_fg$, then $(D_2,v_f,v_l)\in\C_{acc}$ for any $D_2\ge D_1$. This property was exploited in our previous work  to compute in closed form two sets of CBFs for ACC (see \cite{aaron2016barriertac} and its supplemental material \cite{barriersupplemental16}). A set of three or four continuously differentiable functions are provided, with which the CBF $h_{acc}$ is constructed from these functions by \eqref{glueh}.
	In the derivation in \cite{barriersupplemental16}, we assumed that the first equality in \eqref{ACCmodel} is simplified to
	$\dot{v}_f = \frac{\hat u_2}{m},$
	where  $\hat u_2\ge -\hat a_fmg$ for some $\hat a_f>0$. For the ACC subsystem considered here,  we have $\hat u_2 =  u_2 +F_r - m\nu r$, which implies that $\hat a_f  \ge a_f + (c_0+c_1\underline{v}+c_2\underline{v}^2)/mg - \nu_m r_m/g.$
	
	
	In summary, by using a deceleration bound that takes into account the aerodynamic drag and the bound of $|\nu r|$ in the contract, the closed-form CBFs proposed in \cite{aaron2016barriertac} and \cite{barriersupplemental16} are used to construct $h_{acc}(\x_2)$ satisfying properties (ACC-P1)-(ACC-P3).}

\section{Compositional Control Synthesis Via Quadratic Program}\label{sec:composition}

%

In this section, a solution will be provided to the composition problem of LK and ACC formulated in Section \ref{subsec:problem}.

{\color{black}Once the CBFs are obtained (off line), the controls are generated by quadratic programs that combine the hard constraints (i.e., safety), which are expressed as CBF conditions, and the soft constraints (i.e., performance objectives), which indicate closeness to a nominal controller but are overridden when they conflict with the hard constraints.}



The hard constraints are expressed as the controlled invariance of the sets $\C_{lk}$ for LK and $\C_{acc}$ for ACC, using the CBF condition. Because of the assumptions and guarantees between the two subsystems, the controlled invariant set for the compositional system is a Cartesian product  of $\C_{lk}$ and $\C_{acc}$, and the behaviors of the LK and ACC subsystems can be decoupled as long as their states are confined within these two sets, respectively. Therefore, local controllers for LK and ACC can be synthesized by solving two separate QPs.

Based on the CBFs $h_{lk}(\x_1),h_{acc}(\x_2)$ constructed in Section \ref{sec:barrsyn}, the hard constraints for LK and ACC can be expressed  with some positive gains $\gamma_1,\gamma_2$ as follows
\begin{align}
	& L_{f_1+\Delta f_1} h_{lk}(\x_1) + L_{g_1} h_{lk}(\x_1) u_1 +\gamma_1 h_{lk}(\x_1) \geq 0,\label{cbfcon1}\\
	& L_{f_2+\Delta f_2} h_{acc}(\x_2) + L_{g_2} h_{acc}(\x_2) u_2 +\gamma_2 h_{acc}(\x_2)\geq 0.\label{cbfcon2}
\end{align}




The soft constraint \eqref{LK_soft} for LK can be expressed as follows:
\begin{align}\label{clfcon1}
	u_1 & = \bar K(\x_1 - \x_1^{f}) + \delta_1,
\end{align}
where $\delta_1>0$ is a relaxation variable, $\x_1^f = [0, 0, 0, d]^\top$ is a feedforward term, and $\bar K$ is a feedback gain determined by solving a LQR problem such that $\x_1\rightarrow\x_1^f$.

For ACC, we use a candidate CLF
$V(\x_2): = (v_f- v_d)^2$ to express the soft constraint \eqref{eqn:ACCobj} with the following CLF condition:
\begin{align}\label{clfcon2}
	L_{f_2+\Delta f_2}V(\x_2) + L_{g_1}V(\x_2) u_2 + c V(\x_2)  \leq \delta_2,
\end{align}
where $\delta_2>0$ is the second relaxation variable, and $c>0$ is a given constant related to the convergence rate.

Then, among the set of controls that satisfy constraints \eqref{cbfcon1}-\eqref{cbfcon2}, the \emph{min-norm controllers} \cite{FreemanSIAM96} are obtained by solving the following two QPs:

\begin{align}
	\ubar_1^*(x) &= \underset{\ubar_1 = \left[u_1 , \delta_1\right]^\top \in  \R^2}{\operatorname{argmin}}   \frac{1}{2}
	\ubar_1^\top H_{\mathrm{lk}} \ubar_1 + F_{\mathrm{lk}}^\top \ubar_1\tag{QP-LK}\label{QPLK}\\
	\mathrm{s.t.}
	&\quad  A_{lk}  \ubar_1 \leq b_{lk}, \nonumber\\
	&\quad u_1=-K(\x_1-\x_1^{f})+\delta_1, \nonumber \\
	\ubar_2^*(x) &= \underset{\ubar_2 = \left[ u_2,\delta_2\right]^\top \in  \R^2}{\operatorname{argmin}}   \quad \frac{1}{2}
	\ubar_2^\top H_{\mathrm{acc}} \ubar_2 + F_{\mathrm{acc}}^\top \ubar_2 \tag{QP-ACC}\label{QPACC} \\
	\mathrm{s.t.}
	&  \quad A_{acc}  \ubar_2 \leq b_{acc}, \nonumber\\
	&  \quad A^{\mathrm{clf}}_{acc} \ubar_2 \leq b_{acc}^{\mathrm{clf}}+\delta_2, \nonumber
\end{align}
where
%
%
\begin{align*}
	&H_{\mathrm{lk}}  = \left[ \begin{array}{cc} 1 & 0 \\ 0 & p_2  \end{array} \right] , \:\:
	F_{\mathrm{lk}}  =  \left[ \begin{array}{c}  0 \\ 0    \end{array} \right], \\
	&H_{\mathrm{acc}} =  \left[ \begin{array}{cc} \frac{1}{m^2} & 0 \\ 0 & p_1  \end{array} \right] , \:\:
	F_{\mathrm{acc}}  =  - \left[ \begin{array}{c} \frac{\Fr }{m^2}  \\ 0    \end{array} \right],\\
	&A_{lk} = [ - L_{g_1} h_{lk}(\x_1), 0], \\
	&b_{lk} =  L_{f_1+\Delta f_1} h_{lk}(\x_1) + \gamma_1 h_{lk}(\x_1), \\
	&A_{acc} = [ - L_{g_2} h_{acc}(\x_2), 0], \\
	&b_{acc} =  L_{f_2+\Delta f_2} h_{acc}(\x_2) + \gamma_2 h_{acc}(\x_2), \\
	&A_{acc}^{clf} = [  L_{g_2}V(x), -1],\\
	&b_{acc}^{clf} =  - (L_{f_2+\Delta f_2}V(x) + c V(x)) ,
\end{align*}
and $p_1,p_2\gg 0$ are the penalizing weights for relaxation variables $\delta_1, \delta_2$, respectively. Here, $H_{acc},F_{acc}$ are chosen as such due to partial input/output linearization of
\eqref{ACCmodel} \cite{aaron2016barriertac}.

{\color{black}As \eqref{QPLK} and \eqref{QPACC} are convex QPs, they can be solved efficiently by current optimization solvers. Alternatively, it was shown in \cite{aaron2016barriertac} that $u_1,u_2$ obtained by \eqref{QPLK} and \eqref{QPACC} can be obtained in closed-form and are locally Lipshcitz continuous, which makes them particularly easy to use in embedded implementations.}




\begin{theorem}
	The solutions $\ubar_1^*(x)$ and $\ubar_2^*(x)$ generated by \eqref{QPLK} and \eqref{QPACC} constitute a locally Lipshcitz continuous control law that ensures the hard constraints \eqref{eqn:LK_constraint}, \eqref{eqn:LK_constraint2} and \eqref{eqn:ACC_constraint} are satisfied for all time.
\end{theorem}


\begin{remark}
	It is possible to add more performance objectives as soft constraints to the QPs \eqref{QPLK} and \eqref{QPACC}. For instance,
	the soft constraint about the lateral acceleration can be expressed as $|\dot{\nu}|\le \dot{\nu}_{\max}+\delta_3$
	where $\delta_3$ is another relaxation variable and $\dot{\nu}_{\max}$ is the given lateral acceleration bound.
	Adding this constraint and modifying the matrices $H_{lk},F_{lk}$ to  \eqref{QPLK} in an obvious way, we can still obtain a solution that ensures the satisfaction of the hard constraints.
\end{remark}

{\color{black}
	The objective functions of \eqref{QPLK} and \eqref{QPACC} can be expressed alternatively as minimizing the difference of the real control and some nominal control. The corresponding QPs are expressed as follows:
	\begin{align}
		u_1^*(x) &= \underset{u_1 \in  \R}{\operatorname{argmin}} \|u_1-u^{lk}_{no}\|_2
		\tag{QP-LK2}\label{QPLK2}\\
		\mathrm{s.t.}
		&\quad  A_{lk}  u_1 \leq b_{lk}, \nonumber\\
		u_2^*(x) &= \underset{u_2\in  \R}{\operatorname{argmin}} \|u_2-u^{acc}_{no}\|_2
		\tag{QP-ACC2}\label{QPACC2} \\
		\mathrm{s.t.}
		&  \quad A_{acc}  u_2 \leq b_{acc}. \nonumber
	\end{align}
	The nominal control $u^{lk}_{no}$ (resp. $u^{acc}_{no}$) can be either determined by the CLF condition (resp. the LQR solution) shown above or any other legacy control laws such as those have been developed by OEMs. This shows a particular advantage of the proposed control approach in that it can endow a legacy controller as a correct-by-construction solution, where the safety of the closed-loop system is guaranteed by the CBF conditions. Particularly, as QPs \eqref{QPLK2} and \eqref{QPACC2} essentially involve computing a minimum distance to a convex set, their closed-form solutions can be easily obtained \cite{LuenbergerOptimizationVectorSpaceMethods}. Hence, the onboard implementation of our controller is no more
	burdensome than a classic PID controller.}




\section{Simulation}\label{sec:simu}
%
%
%


{\color{black}In this section, we apply the control laws $u_1,u_2$, which are obtained by solving QPs \eqref{QPLK} and \eqref{QPACC}, to the simultaneous operation of LK and ACC in a 16 degree of freedom model in Carsim, which is  a widely used vehicle simulation package in industry. Although the controllers are designed by the widely used simplified models \eqref{LKmodel} and \eqref{ACCmodel}, the overall system is shown to satisfy all of the safety specifications.
	
	
	
	The parameter values related to the vehicle dynamics are extracted from the “D-Class Sedan” model in CarSim, and are shown in Table \ref{tab:para}.
	The controlled car is allowed to employ a maximal deceleration of $0.25~g$, maximal steering angle of $0.06~rad$ (i.e., approximately 3.5 degrees), and have a desired pre-set speed $v_d=22~m/s$ and time-headway setting $\tau_d=1.8$ seconds. The bound $d_{\max}$ related to the road curvature is given as $0.1~rad/s$, and the bounds related to the lateral dynamics are given as $y_{\max}=0.9~m$,  $\nu_{\max}=1~m/s$, $\Delta\psi_{\max}=0.05~rad$ and $r_{\max}=0.3~rad/s$, respectively.
}

\begin{table}[!bth]
	\caption{Parameter values and constraints}\label{tab:para}
	\begin{center}
		\begin{tabular}{|c|c||c|c||c|c|}
			\hline
			$m$  & 1650 $kg$ & $y_m$ & 0.9 $m$ & $p_1$ &  1000\\
			$c_0$ & 51 $N$ & $\nu_m$ & 1.0 $m/s$ & $p_2$ & 1000\\
			$c_1$ & 1.26 $N  s/m$ &$\Delta\psi_m$ & 0.05 $rad$ & $p_3$ & 100 \\
			$c_2$ & 0.4342 $N  s^2/m^2$ & $r_m$ & 0.3 $rad/s$ & $a_f$ & 0.25\\
			$a$ & 1.11	$m$ & $\underline{v}$ & 15 $m/s$ & $a_f'$ &  0.25\\
			$b$ & 1.59 $m$ & $\bar{v}$    & 30 $m/s$ & $a_l$ &  0.25\\
			$C_{f}$ & $133000$ $N/rad$ & $v_{\min}$ & 15 $m/s$ &  $\tau_d$ & 1.8\\
			$C_{r}$ & $98800$ $N/rad$ & $v_{\max}$ & 30 $m/s$ & $\gamma_2$& 2 \\
			$I_z$ & 2315.3 $kg\; m^2$ $rad/s$ &$v_d$ & 22 $m/s$ &  $\gamma_1$& 2 \\
			$d_{\max}$ & 0.1 & $D_0$ & 0.1 $m$ &$c$ & 10 \\
			$\dot{\nu}_{\max}$ & 0.25 $m/s^2$ & $g$ & 9.81 $m/s^2$ & $\hat\delta_f$ & 0.06  \\
			\hline
		\end{tabular}
	\end{center}
\end{table}


{\color{black}With the given parameters, the CBF $h_{lk}(\x_1)$ is obtained by solving the SOS programs using the MATLAB toolbox yalmip along with the SDP solver Mosek. For ACC, the set of \emph{optimal barriers} $h_{acc}(\x_2)$  developed  in \cite{barriersupplemental16} are used.  The QPs are solved using the MATLAB command \textit{quadprog}, but they could be just as easily solved in closed-form as mentioned in Section \ref{sec:composition}.
	The gain matrix $\bar K$ in \eqref{clfcon1} is obtained by solving an LQR problem. Assume that the controlled car uses a lateral preview of approximately $0.4$ seconds, which corresponds to an ``output'' $Cx$ with $C=[1, 0, 10, 0]$.  Given the control weight $R=600$ and the state weight matrix $Q= K_p C^\top C + K_d C^\top A_1^\top A_1 C$,
	where $A_1$ is given in \eqref{PVLK} and $K_p=5,K_d=0.4$,
	the feedback gain $\bar K$ is determined by solving an LQR problem with such $Q$ and $R$.

	
	
	
	We simulated the controller in Carsim on a curved road, with initial condition $(v_f(0),v_l(0),D(0))=(18,17,65)$ and $(y(0),\nu(0),\Delta\psi(0),r(0))=(0,0,0,0)$. Each subfigure of Figure \ref{fig:simu} is interpreted as follows:
	\begin{itemize}
		\item
		Subfigure (a) shows the speed profile of the lead car $v_f$ (in blue) and the controlled car $v_l$ (in black), where the desired speed $v_d$ is depicted in dotted green line.
		During $t=0s$ and $t=7s$, the controlled car accelerates and achieves the desired speed $v_d$;
		during $t=7s$ and $t=12s$, the controlled car slows down to maintain a safe distance with the lead car; during $t=12s$ and $t=24s$, the controlled car achieves the desired speed again; after $t=24s$, the controlled car slows down to keep a safe distance with the lead car and eventually maintains the same speed as the lead car. The ``apparent changes in modes'' are all achieved by the QP guaranteeing the  hard (safety) constraint and meeting the soft (performance) constraint on speed tracking as closely as possible. There are no if-then-else or case statements involved.
		\item
		Subfigure (b) shows the state evolution of the lateral dynamics $\x_2$; it can be seen that $y,\nu,\Delta\psi,r$ are within their respective bounds during the simulation. The abrupt changes happening at $t=20s,25s,45s$ are due to step changes in the road curvature. 
		\item
		Subfigure (c) shows the wheel force $u_1$ divided by $mg$ (in blue) and the bounds $-a_f,a_f'$ (in dotted red line); it can be seen that the wheel force constraint is satisfied as $u_1\in U_{acc}$.
		\item
		Subfigure (d) shows the steering angle $u_2$ (in blue) and the bound $\hat\delta_f,-\hat\delta_f$ (in dotted red line); it can be seen that the steering angle constraint is satisfied as $u_2\in U_{lk}$.
		\item
		Subfigure (e) shows the values of $r$ (in black) and $d$ (in blue); it can be seen that $r$ tracks $d$ pretty well, indicating satisfaction of the soft constraint \eqref{eqn:LK_constraint2}.
		\item
		Subfigure (f) shows the values of the actual time headway $\tau$ (in blue) and the desired time headway $\tau_d$ (in dotted red line); it can be seen that $\tau\geq \tau_d$ as desired.
		\item
		Subfigure (g) shows the  values of $h_{acc}$ (in blue) and the zero-value line (in dotted red); it can be seen that $h_{acc}$ is always positive as desired, which implies satisfaction of the ACC constraint \eqref{eqn:ACC_constraint}. Note that $h_{acc}>0$ during $t\in[12s,24s]$ because the ACC-controlled car has achieved its desired speed and has no intension to reduce the relative distance.
		\item
		Subfigure (h) shows the values of $h_{lk}$ (in blue) and the zero-value line (in dotted red); it can be seen that $h_{lk}$ is always positive as desired, which implies satisfaction of the LK constraints \eqref{eqn:LK_constraint} and \eqref{eqn:LK_constraint2}.
	\end{itemize}}
	
	\begin{figure*}[!hbt]
		\begin{center}
			\subfigure[]{\includegraphics[width=0.41\textwidth]{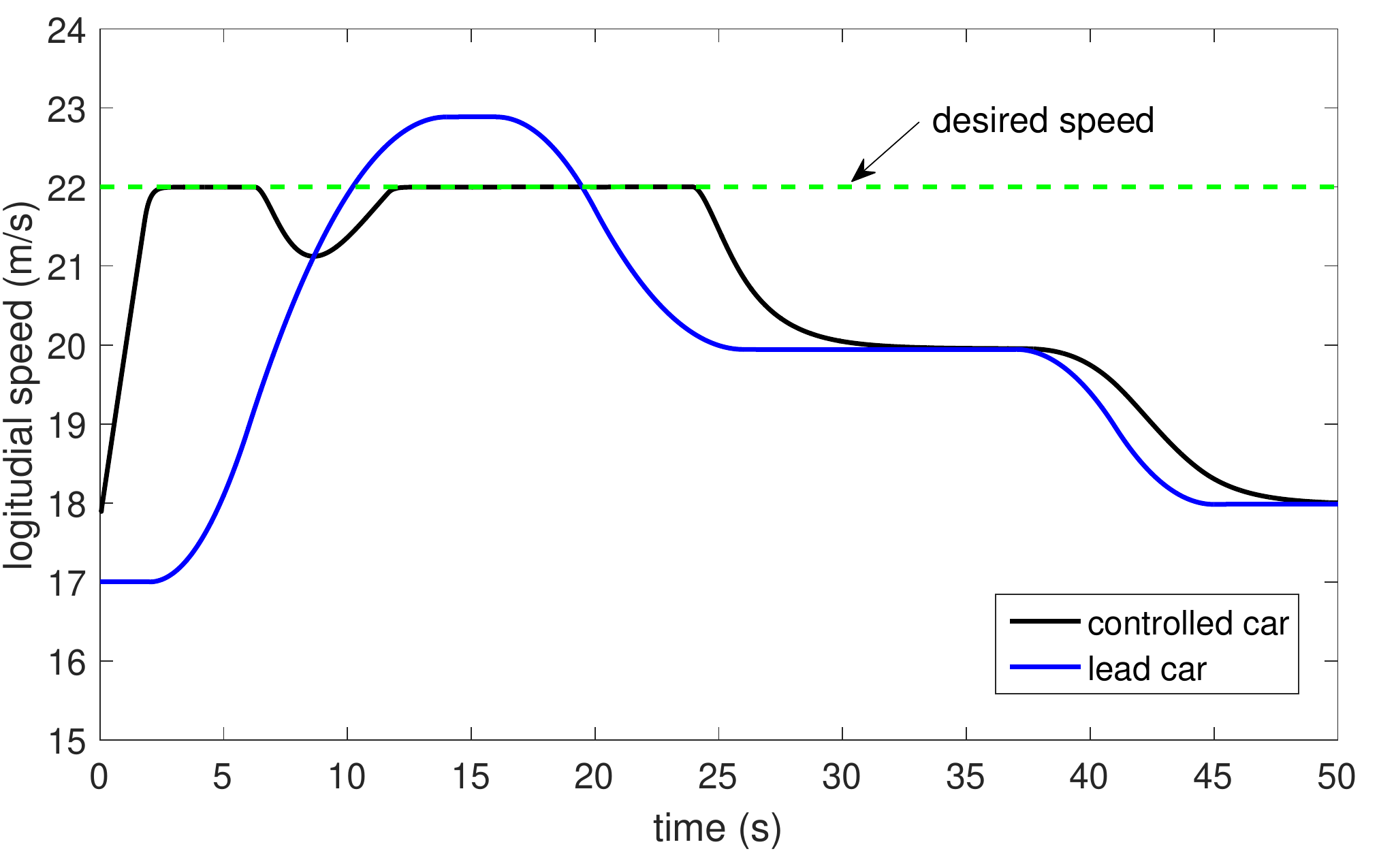}}\hskip 8mm
			\subfigure[]{\includegraphics[width=0.41\textwidth]{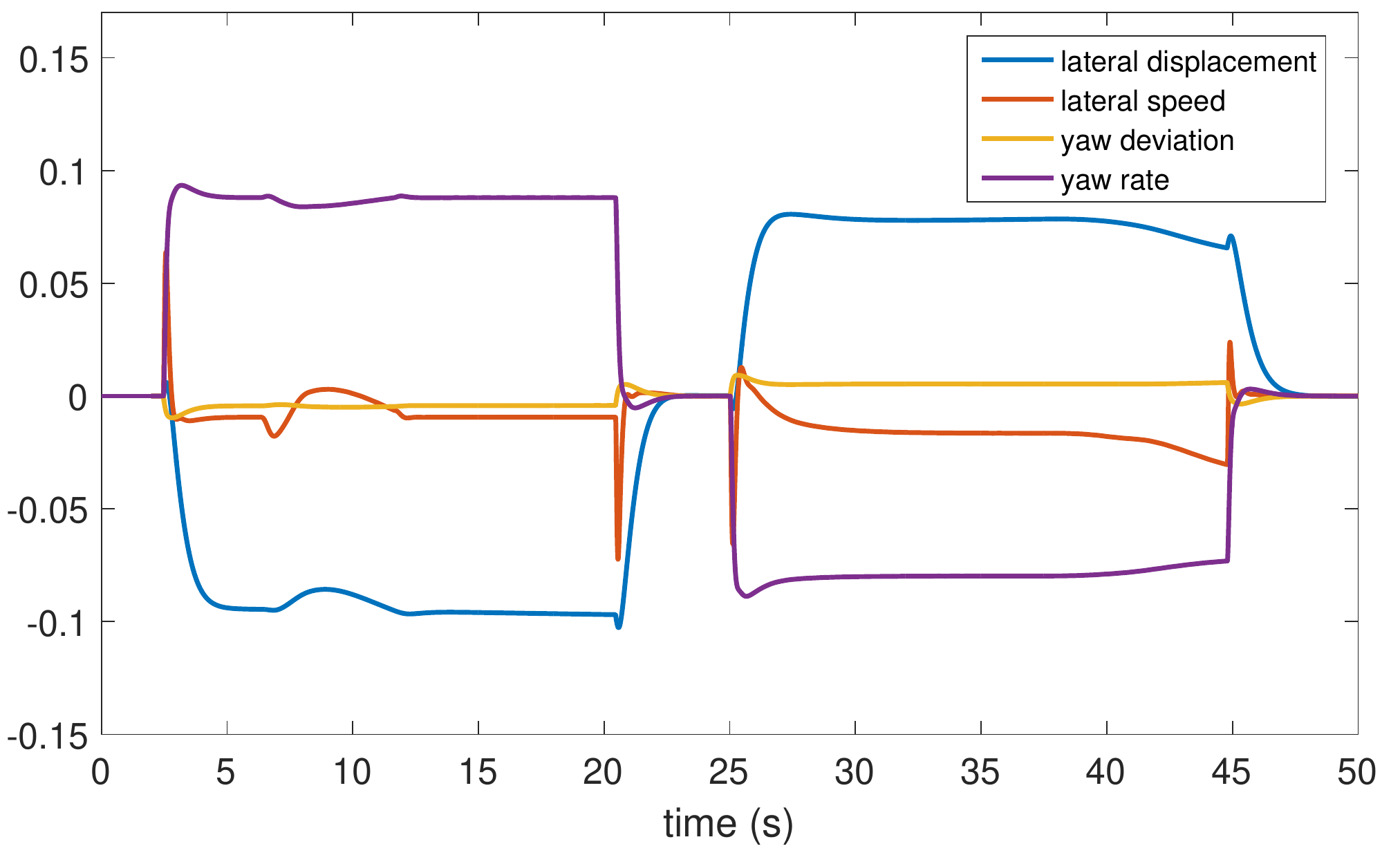}}
			\subfigure[]{\includegraphics[width=0.41\textwidth]{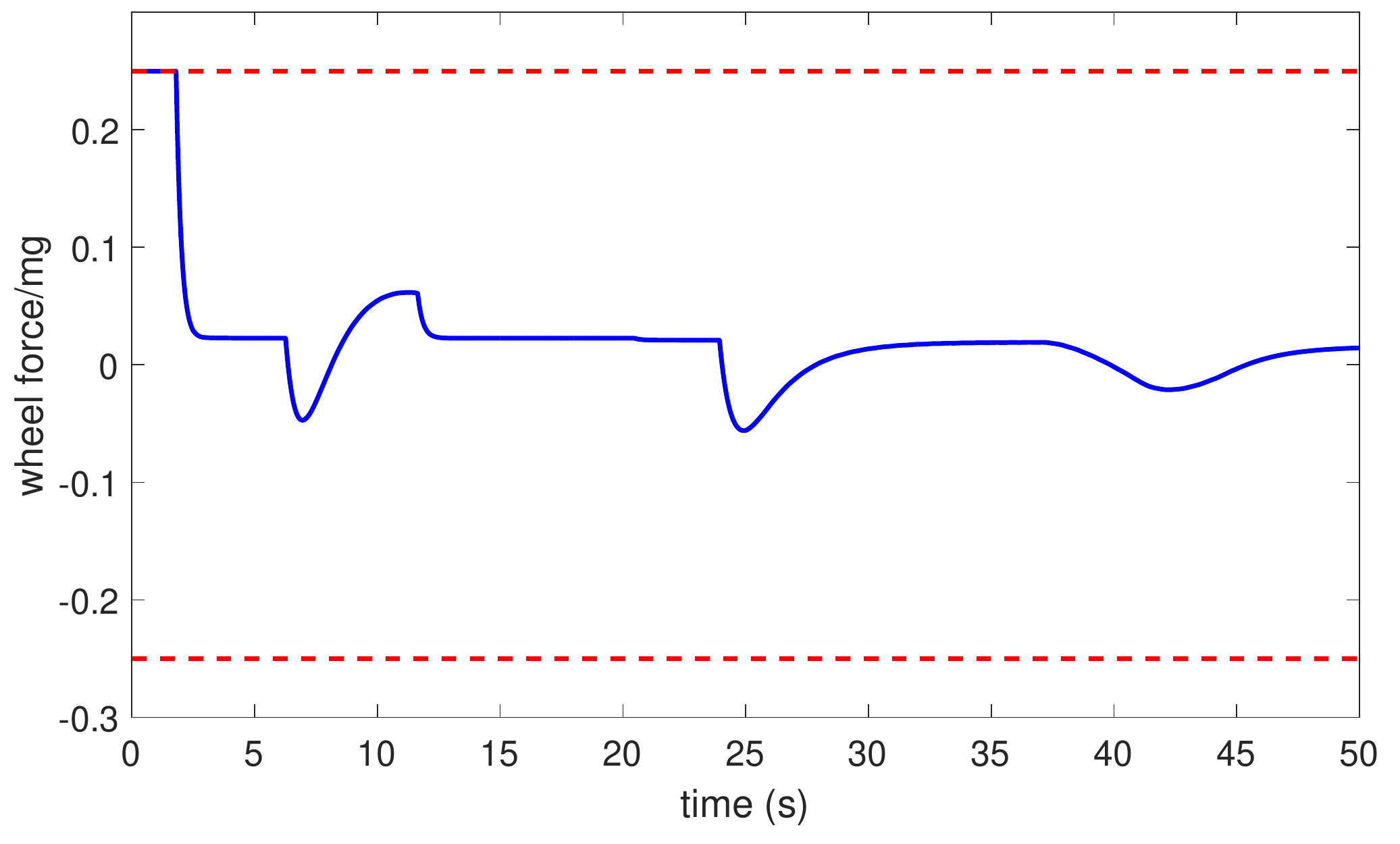}}\hskip 8mm
			\subfigure[]{\includegraphics[width=0.41\textwidth]{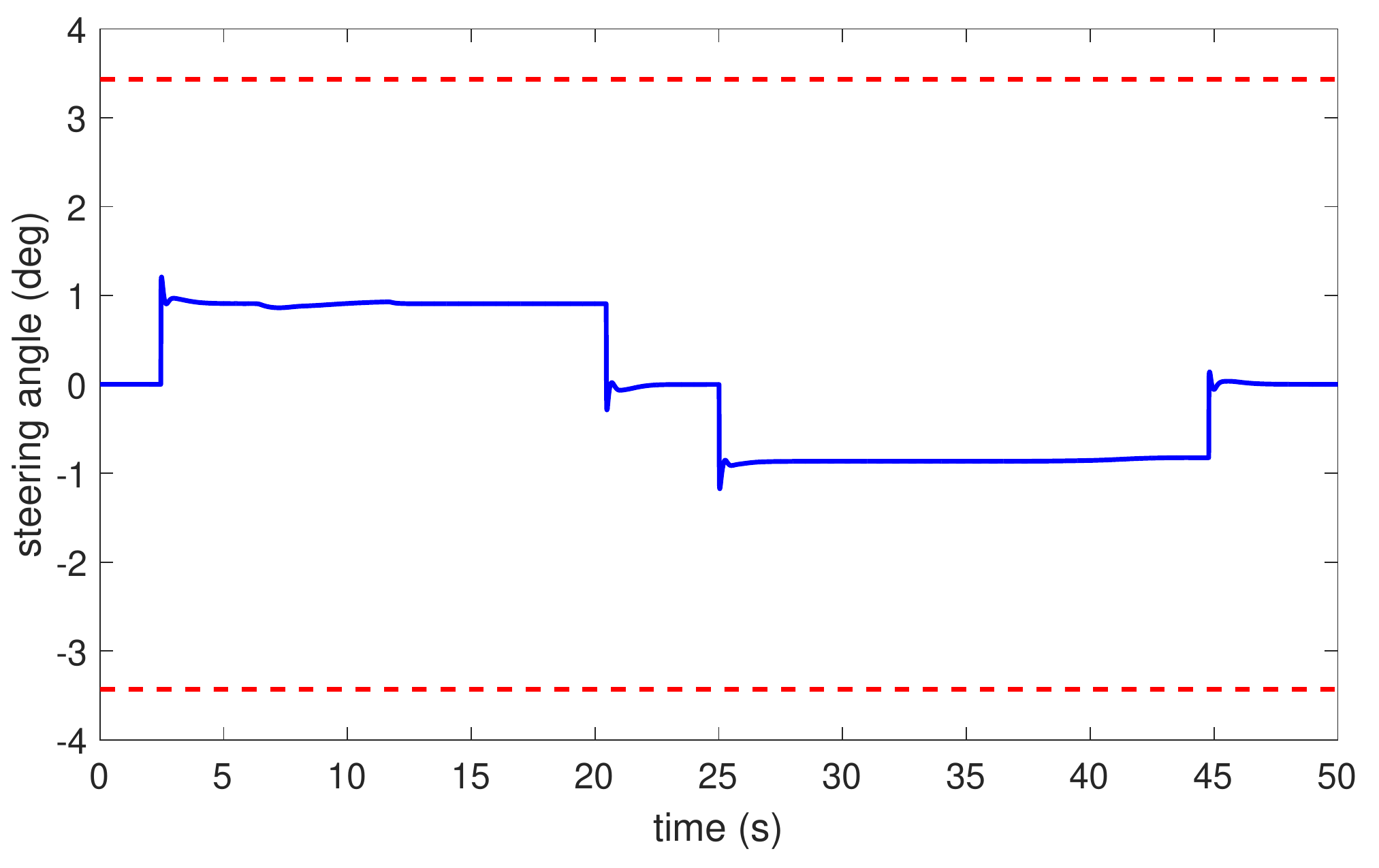}}
			\subfigure[]{\includegraphics[width=0.41\textwidth]{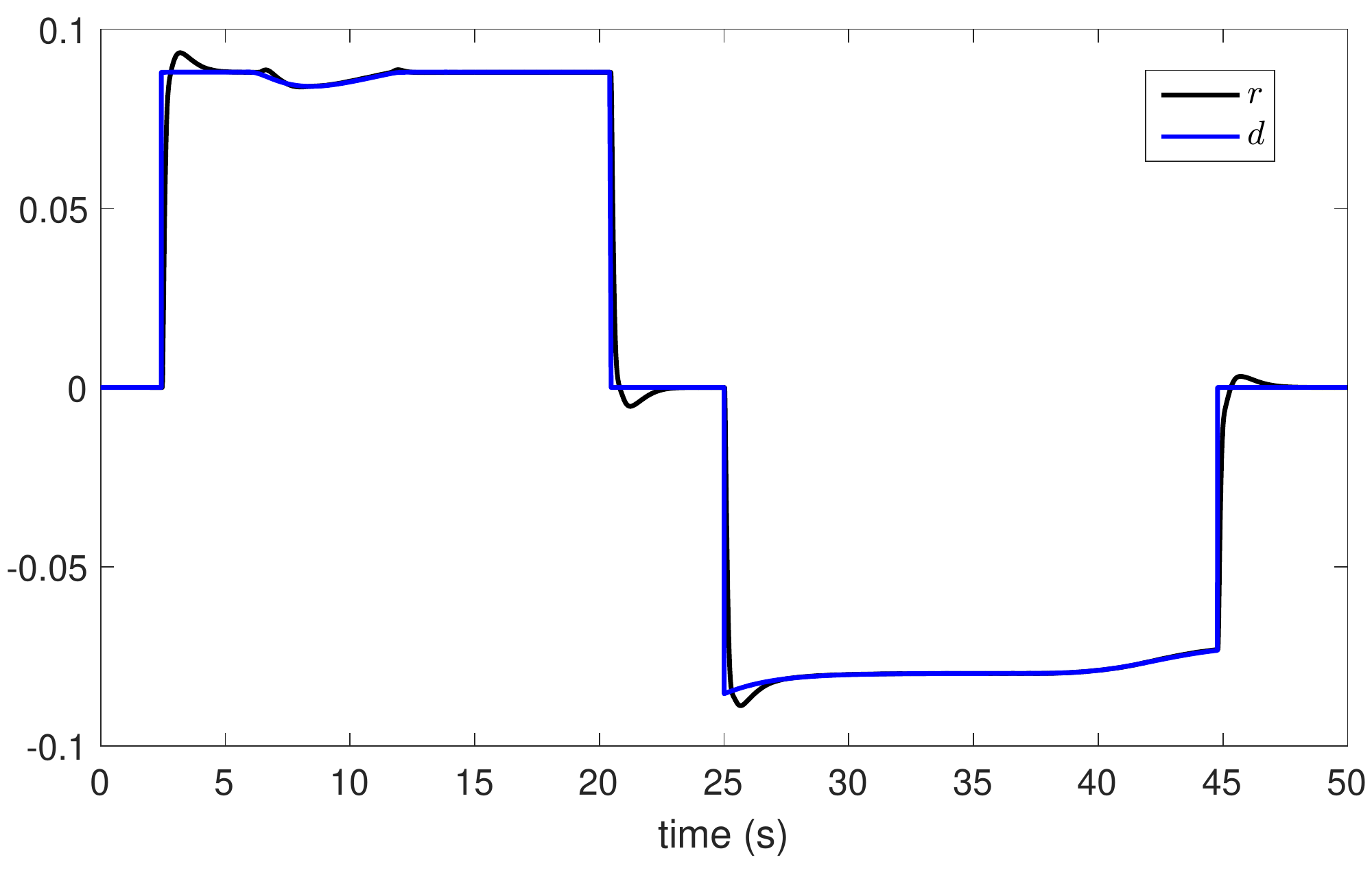}}\hskip 8mm
			\subfigure[]{\includegraphics[width=0.41\textwidth]{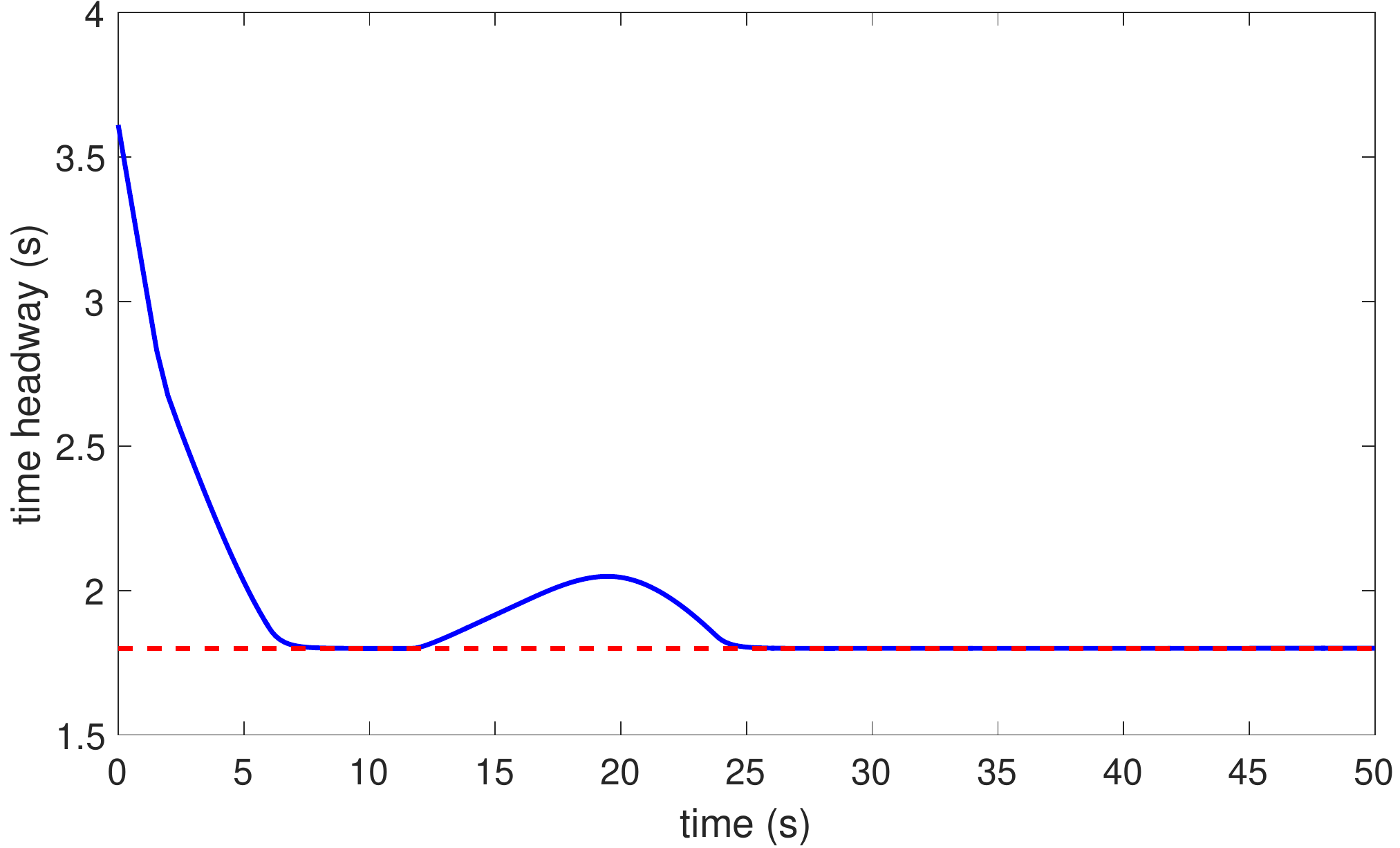}}
			\subfigure[]{\includegraphics[width=0.41\textwidth]{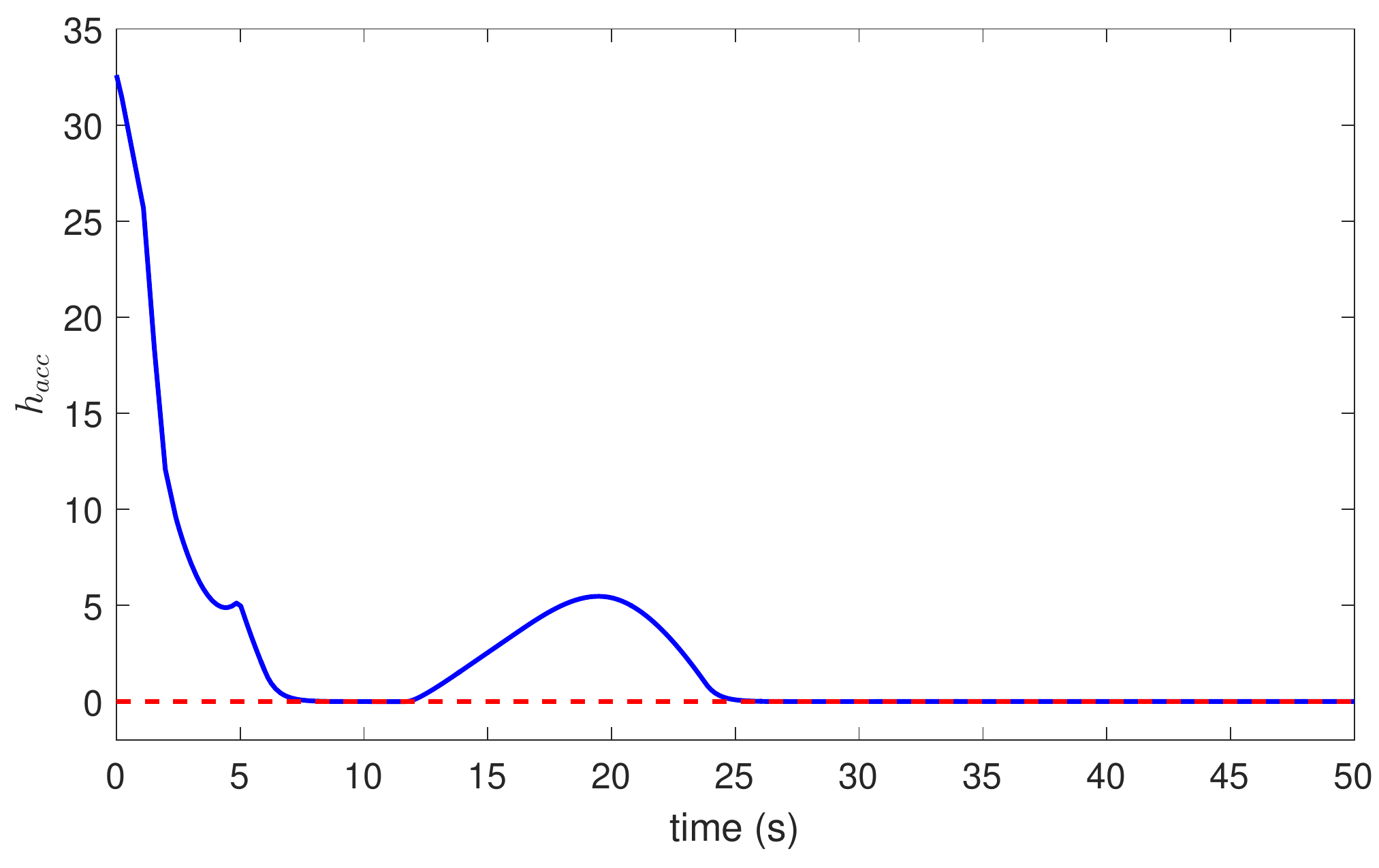}}\hskip 8mm
			\subfigure[]{\includegraphics[width=0.41\textwidth]{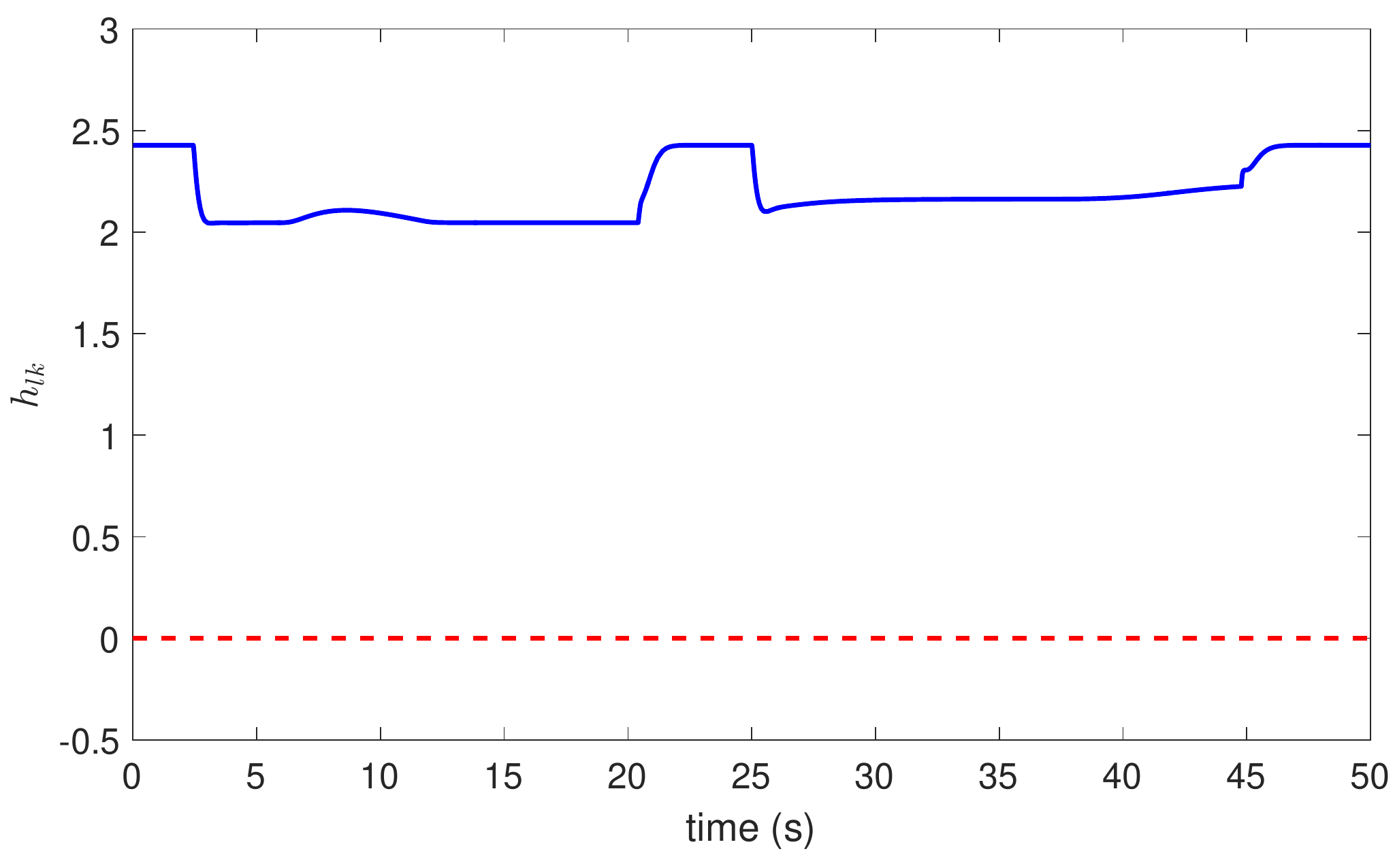}}
			\caption{(a) Speed of the controlled car $v_f$ (in black), speed of the lead car $v_l$ (in blue), and the desired speed $v_d$ (in dotted green). (b) State evolution of the lateral dynamics $y,\nu,\Delta\psi,r$.
				(c) The wheel force $u_1$ divided by $mg$ (in blue) and its bound $\pm 0.25$ (in dotted red). (d) The steering angle $u_2$ and its bound $\pm 3.4$ deg (in dotted red). (e) Values of $r$ (in black) and $d$ (in blue). (f) Values of the actual time headway (in blue) and the desired time headway (in dotted red) (g) Values of the CBF  $h_{acc}$ (in blue), where non-negativeness implies satisfaction of the ACC constraint \eqref{eqn:ACC_constraint}. (h) Values of the CBF  $h_{lk}$ (in blue), where non-negativeness implies satisfaction of the LK constraints \eqref{eqn:LK_constraint} and \eqref{eqn:LK_constraint2}.}\label{fig:simu}
		\end{center}
	\end{figure*}

\section{Conclusions}\label{sec:conclusion}

{\color{black} In this paper, we developed a control design approach with correctness guarantees for the simultaneous operation of lane keeping and adaptive cruise control, where the longitudinal force and steering angle are generated by solving quadratic programs. The safety constraints are hard constraints that are enforced by confining the states of the vehicle within determined controlled-invariant sets, which are expressed as CBF conditions. The performance objectives are soft constraints that can be overridden when they are in conflict with safety. The proposed QP-based framework can integrate legacy controllers as the performance controller and endow them with correct-by-construction solutions that guarantee safety. Additionally, the QP solution is known in closed form, and thus the proposed algorithm can be implemented without online optimization, if desired. The effectiveness of the proposed controller is shown by simulations in Carsim.

	The SOS algorithm used to construct CBFs for the lane keeping is quite general and can be applied to other safety control problems as well. The assume-guarantee formalism is well adapted to modularity of the driver assistance modules studied in this paper because any control laws respecting the contracts given for lane keeping and adaptive cruise control, respectively, will guarantee safety of the closed-loop system when the two modules are activated simultaneously. This means in particular that the individual modules do not have to be provided by the same supplier as long as an OEM provides the correct contracts.
	
	A preliminary test of the LK and ACC algorithms has been conducted on the Khepera robot and the Robotarium testbed \cite{xuccta2017}. Our future plans include testing these algorithms on a full-sized vehicle. Additional challenges include validating the models used for control design and considering sensor errors when constructing CBFs.}

%
%


\bibliographystyle{IEEEtran}
\bibliography{./composite_references,./new_ref,./barrier,./SOS,./smallgain,./Rfunction,./vehicle,./invariance_control,./literature_review,./reference_governor}

\end{document}